\documentclass[preprint]{imsart}

\RequirePackage[OT1]{fontenc}
\RequirePackage{amsthm,amsmath}
\RequirePackage[numbers]{natbib}
\usepackage{mathrsfs}
\usepackage{comment}
\usepackage{graphicx}
\usepackage{caption}
\usepackage{subcaption}
\startlocaldefs
\theoremstyle{plain}

\newtheorem*{thm*}{Theorem}

\newtheorem{lem}{Lemma}
\newtheorem*{lem*}{Lemma}

\newtheorem*{defi*}{Definition}

\newtheorem*{rques*}{Remarks}
\DeclareMathOperator{\sign}{sign}
\newcommand{\mb}{\mathbb}
\newcommand{\Tr}{\operatorname{Tr}}

\newcommand{\diag}{\operatorname{diag}}

\newcommand{\E}{{\bf E}}
\newcommand{\R}{{\mathcal R}}
\endlocaldefs

\usepackage{ulem}
\usepackage{graphicx}
\usepackage{enumitem}

\newtheorem{theorem}{Theorem}[section]
\newtheorem{corollary}{Corollary}[section]
\newtheorem{prop}{Proposition}[section]

\newcommand{\simo}[1]{{\color{red}#1}}
\newcommand{\yiqiu}[1]{{\color{blue}#1}}

\usepackage{amsfonts,amssymb,amsthm,amsmath}
\usepackage{stmaryrd}
\usepackage{hyperref}

\usepackage{makecell}
\usepackage{capt-of}
\usepackage{dsfont}
\usepackage{tikz}
\usepackage{multirow}
\usepackage{tabularx}
\usepackage{url}

\allowdisplaybreaks

\usepackage{xifthen}
\usepackage{xargs}

\usepackage{mathtools}

\usepackage{longtable}
\usepackage{tabularx}
\usepackage{mathrsfs}
\usepackage{multirow}

\usepackage{calrsfs}

\usepackage{array}
\usepackage{makecell}

\usepackage{thmtools,thm-restate}

\begin{document}

\begin{frontmatter}
\title{Minimax Supervised Clustering in the Anisotropic Gaussian Mixture Model: A new take on Robust Interpolation}
\runtitle{ Clustering in the Anisotropic Gaussian Mixture Model}
%\thankstext{T1}{Footnote to the title with the ``thankstext'' command.}

%\begin{aug}
%%%%%%%%%%%%%%%%%%%%%%%%%%%%%%%%%%%%%%%%%%%%%%%
%% Only one address is permitted per author. %%
%% Only division, organization and e-mail is %%
%% included in the address.                  %%
%% Additional information can be included in %%
%% the Acknowledgments section if necessary. %%
%%%%%%%%%%%%%%%%%%%%%%%%%%%%%%%%%%%%%%%%%%%%%%%
%\author[A]{\fnms{???} \snm{???}\ead[label=e1]{???@???}},
%\author[B]{\fnms{???} %%\snm{???}\ead[label=e2,mark]{???@???}}
%\and
%\author[B]{\fnms{???} %\snm{???}\ead[label=e3,mark]{???@???}}
%%%%%%%%%%%%%%%%%%%%%%%%%%%%%%%%%%%%%%%%%%%%%%
%% Addresses                                %%
%%%%%%%%%%%%%%%%%%%%%%%%%%%%%%%%%%%%%%%%%%%%%%
%\address[A]{???, \printead{e1}}

%%\address[B]{???, \printead{e2,e3}}
%\end{aug}

\begin{aug}
\author[A]{\fnms{Stanislav} \snm{Minsker}\ead[label=e2]{minsker@usc.edu}},
\author[B]{\fnms{Mohamed} \snm{Ndaoud}
\ead[label=e3]{ndaoud@essec.edu}}
\and
\author[C]{\fnms{Yiqiu} \snm{Shen}\ead[label=e1]{yiqiushe@usc.edu}}
%\thankstext{t1}{Some comment}
%\thankstext{t2}{First supporter of the project}
%\thankstext{t3}{Second supporter of the project}
%\runauthor{F. Author et al.}

%\affiliation{University of Southern California\thanksmark{m1} and ESSEC Business School\thanksmark{m2}}

\address[A]{Department of Mathematics\\
University of Southern California\\
Los Angeles, CA 90089 \\
\printead{e2}}

\address[B]{Department of Information Systems, Decision Sciences and Statistics\\
ESSEC Business School\\
95000 Cergy, France\\
\printead{e3}}

\address[C]{Department of Data Sciences and Operations\\
University of Southern California\\
Los Angeles, CA 90089\\
\printead{e1}}
%\phantom{E-mail:\ }\printead*{e2}}

\end{aug}

\begin{abstract}
We study the supervised clustering problem under the two-component anisotropic Gaussian mixture model in high dimensions and in the non-asymptotic setting. We first derive a lower and a matching upper bound for the minimax risk of clustering in this framework. We also show that in the high-dimensional regime, the linear discriminant analysis (LDA) classifier turns out to be sub-optimal in the minimax sense. Next, we characterize precisely the risk of $\ell_2$-regularized supervised least squares classifiers.  We deduce the fact that the interpolating solution may outperform the regularized classifier, under mild assumptions on the covariance structure of the noise. Our analysis also shows that interpolation can be robust to corruption in the covariance of the noise when the signal is aligned with the ``clean'' part of the covariance, for the properly defined notion of alignment. To the best of our knowledge, this peculiar phenomenon has not yet been investigated in the rapidly growing literature related to interpolation. We conclude that interpolation is not only benign but can also be optimal, and in some cases robust.
\end{abstract}

%\begin{keyword}[class=MSC]
%\kwd[Primary ]{60K35}
%kwd{60K35}
%\kwd[; secondary ]{60K35}
%\end{keyword}

%\begin{keyword}
%\kwd{sample}
%\kwd{\LaTeXe}
%\end{keyword}

\end{frontmatter}

\section{Introduction}
The topic of overparametrization has gained tremendous attention in recent literature devoted to the problems in high dimensional statistics. Previously, it was widely believed that regularization yields the best generalization power. Recently, it was discovered that estimators that interpolate the training data also yield good generalization error bounds when the number of covariates exceeds the sample size. This phenomenon, termed ``benign overfitting'' in \cite{bartlett2020benign}, has been extensively investigated in the regression setting. In this work, we study the problem of clustering. In particular, we derive the bounds for the generalization error under different sets of assumptions. The model we consider is a binary sub-Gaussian mixture model with unknown, anisotropic noise.

\subsection{Statement of the problem}
Consider the simple two-component Gaussian mixture model, where we observe for all $i=1,\dots,n$ the pairs $(Y_i,\eta_i)$ such that
\begin{equation}\label{eq:model}
Y_i = \theta\eta_i + W_i
\end{equation}
where $\theta \in \mb R ^p$ is a center vector, $\eta = (\eta_1,\ldots,\eta)^T \in \{-1,1\}^n$ a label vector and $W$ a random matrix with i.i.d columns $W_i, \ i=1,\ldots,n$ that are sub-Gaussian. More precisely, given the spectral decomposition $\Sigma = V \Lambda V^\top$, we assume that
\[
W_i = V \Lambda^{1/2} w_i,
\] 
where $w_i$ has components that
are independent and 
$1$-sub-Gaussian; that is, for all $\lambda \in \mathbb{R}^{p}$,
\[
\mathbf{E}(\exp(\lambda^\top w_1)) \leq \exp(\|\lambda\|^2 /2)
\] 
where $\|\cdot\|$ denotes the Euclidean norm. Moreover, we will always assume that $\Sigma$ has full rank.
%or, equivalently, for all $i=1,\dots,n$
%\[
%\eta_i Y_i = \theta  + \eta_i W_i.
%\] 
We mostly focus on the supervised setting, where we are given a classifier $\hat{\eta}$ based on the training data $(\mathbf{Y},\mathbf{\eta}) = (Y_1,\eta_1),\ldots,(Y_n,\eta_n)$ and a new independent observation $(Y_{n+1},\eta_{n+1})$ such that $\eta_{n+1}$ is a Rademacher random variable taking values $\pm 1$ with probability $1/2$ each. We want to analyze the generalization error defined as
\[
\mathcal{R}_{\Sigma}(\hat{\eta}) := \mb P\left( \hat{\eta}((\mathbf{Y,\eta});Y_{n+1}) \neq \eta_{n+1}  |{(\mathbf{ Y,\eta})}\right),
\]
both in probability and expectation, where $\mb P$ is the probability under the model described above. Observe that
\[
\mathbf{E}(\mathcal{R}_{\Sigma}(\hat{\eta})) = \mb P\left( \hat{\eta}((\mathbf{Y,\eta});Y_{n+1}) \neq \eta_{n+1}  \right).
\]
When there is no ambiguity, we will omit the subscript $\Sigma$ from $\mathcal{R}_{\Sigma}(\hat{\eta})$. In particular, we want to analyze the minimax risk
\[
\underset{\hat{\eta}}{\inf}\underset{\|\theta\| \geq \Delta }{\sup} \mathbf{E}(\mathcal{R}_{\Sigma}(\hat{\eta})),
\]
and study the necessary and sufficient conditions on $\Delta$ for consistent clustering, i.e. conditions on $(\Delta_n)_n$ such that
\[
\underset{\hat{\eta}}{\inf}\underset{\|\theta\| \geq \Delta_n }{\sup}\mathbf{E}(\mathcal{R}_{\Sigma}(\hat{\eta})) \underset{n \to \infty}{\to} 0.
\]
The case when $\Sigma = \mathbf{I}_p$ was investigated in \cite{ndaoud2018sharp}. In particular, it was shown that
\[
\underset{\hat{\eta}}{\inf}\underset{\|\theta\| \geq \Delta }{\sup} \mathbf{E}(\mathcal{R}_{\Sigma}(\hat{\eta})) \approx \exp\left(-(1+o_n(1))\frac{\Delta^4}{2(\Delta^2 + \frac{p}{n})}\right).
\]
When $\Sigma$ is known and the noise is normal, is it easy to see that $\Sigma^{-1/2}Y$ follows the isotropic Gaussian Mixture Model where the signal vector is given by $\Sigma^{-1/2}\theta$. Following the reasoning similar to \cite{ndaoud2018sharp}, we can show in the general case that
\[
\underset{\hat{\eta}}{\inf}\underset{\|\theta\|_{\Sigma} \geq \Delta }{\sup} \mathbf{E}(\mathcal{R}_{\Sigma}(\hat{\eta}))  \approx \exp\left(-(1+o_n(1))\frac{\Delta^4}{2(\Delta^2 + \frac{p}{n})}\right),
\]
where $\|\theta\|^2_{\Sigma} = \theta^\top \Sigma^{-1} \theta$. In particular, the condition $\|\theta\|^2_{\Sigma} \gg \sqrt{p/n} + 1$ is necessary and sufficient for consistency under the norm $\|\cdot\|_\Sigma$. 
Moreover, the minimax optimal classifier is the Linear Discriminant Analysis (LDA) classifier given by 
\[
\hat{\eta}_{\text{LDA}}(y) = \sign\left( \left\langle \Sigma^{-1}\sum_{i=1}^n Y_i \eta_i,y\right\rangle \right).
\]
The LDA estimator and its adaptive variants have been recently studied in several works, including \cite{wang2020efficient,davis2021clustering, chen2021optimal}  and \cite{linjun}. All these papers consider anisotropic mixtures in the low dimensional case. In this project we are interested in the case of unknown $\Sigma$ in high dimensions (i.e. $p \gg n$) where estimation of both $\theta$ and $\Sigma$ becomes challenging. 

We are also interested in the supervised risk $\R(\hat{\eta})$ of certain classifiers $\hat{\eta}$ of the form 
\[
\hat{\eta} := \sign\left( \hat{\theta}^{\top} Y_{n+1} \right),
\]
where $\hat{\theta}$ describes a separating half-space. More precisely, we will focus on the minimum $\ell_2$-norm solutions of the optimization problems corresponding either to the Ordinary Least Squares (OLS) or the Support Vector Machines (SVM). The hard margin SVM classifier  $\hat{\theta}_{\text{SVM}}$ is the solution to the problem
\begin{equation}\label{eq:SVM}
\hat{\theta}_{\text{SVM}} = \arg\min_{\theta\in\mb{R}^p} \|\theta\|^2 \text{ subject to } \eta_i \theta^\top Y_i\geq 1, \quad \forall i=1,\dots,n.
\end{equation}
Similarly, $\hat{\theta}_{\text{OLS}}$ is defined as the solution to 
\begin{equation}\label{eq:OLS}
\hat{\theta}_{\text{OLS}} = \arg\min_{\theta\in\mb{R}^p} \|\theta\|^2 \text{ subject to } \eta_i \theta^\top Y_i = 1, \quad \forall i=1,\dots,n.
\end{equation}
While SVM is more commonly used for classification, both approaches have attracted a lot of attention recently, in part due to the discovery \cite{,hsu2020proliferation} of the phenomenon known as  ``Proliferation of Support Vectors'' (SVP). Specifically, SVP corresponds to the situation when $\hat{\theta}_{\text{SVM}}=\hat{\theta}_{\text{OLS}}$, and occurs typically in the high-dimensional setting.  
%\begin{defi*}(Proliferation of Support Vectors)
%We refer to the situation where  as  ``proliferation of support vectors.''
%\end{defi*}

\subsection{Notation and Definitions}
For $\theta\in \mathbb{R}^p$, we denote its Euclidean norm by $\|\theta\|$ and its Mahalanobis norm corresponding to a positive definite matrix $\Sigma$ by $\|\theta\|_\Sigma:=\theta^\top\Sigma^{-1}\theta$. For a symmetric matrix $A \in\mathbb{R}^{p\times p}$ with eigenvalues $\lambda_1\geq \dots \geq \lambda_p$, we define the spectral norm of $A$ via $\|A\|_\infty:=\lambda_1$ and the Frobenius norm by $\|A\|_F:=\sqrt{\sum_{i=1}^p \lambda_i^2}$. For a positive semidefinite (PSD) matrix $A=\sum_{i=1}^p\lambda_iv_iv_i^\top \neq 0_p, A\in\mathbb{R}^{p\times p}$, we define its effective rank by $r(A):= \Tr(A)/\|A\|_{\infty}$ and its $k$-effective rank by $r_k(A):=\frac{\sum_{i=k+1}^p\lambda_i}{\lambda_{k+1}}$. Moreover, we define the orthogonal projectors  $\pi_{k}(A):=\sum_{i=1}^kv_iv_i^\top$, where $\lambda_1\geq\lambda_2\geq\dots\geq\lambda_p$ is the non-increasing sequence of eigenvalues of $A$. For a given sequences $\{a_n\}_{n\geq 1}$ and $\{b_n\}_{n\geq 1}$, we say that $a_n=\Omega(b_n)$ (resp. $a_n=\mathcal{O}(b_n)$) if for some $c>0$, $a_n\geq cb_n$ (resp. $a_n\leq cb_n$) for all integers $n$ large enough. We will also write $a_n=o(b_n)$ if $a_n / b_n \to 0$ as $n$ goes to $\infty$.

\subsection{Related work}

Recent papers, for instance \cite{belkin2018overfitting} and \cite{belkin2019does}, suggest that interpolating solutions (specifically, solutions that fit the training data perfectly) can achieve optimal rates for the problems of nonparametric regression and $k$-nearest neighbour clustering respectively. Termed ``benign overfitting" by \cite{bartlett2020benign}, this phenomenon has been studied analytically in the framework of linear regression with isotropic noise. It was extended to the anisotropic  case later in \cite{wu2020optimal}, where the authors proposed the fruitful idea of ``alignment" and ``misalignment" between the signal and the covariance matrix of the noise. 

There is a number of recent works exploring the subject of overfitting in regression, while our focus is on the derivation of tight bounds for the misclassification rate for clustering in the framework of anisotopic sub-Gaussian mixture models. In  moderate dimensions ($p \leq n$), this problem has been studied extensively, as in the most recent works  \cite{wang2020efficient,davis2021clustering} and \cite{chen2021optimal}, and efficient approaches such as perturbed gradient descent and modifications of Lloyd's algorithm have been proposed. When overparametrization is possible (corresponding to the case $p \geq n$), support vector proliferation \cite{muthukumar2020classification} has emerged as an important aspect of the analysis of interpolating SVM classifiers. In particular, papers including \cite{ardeshir2021support, wang2020benign} establish sufficient conditions for consistency of SVM interpolation. On the other hand, the work \cite{liang2021interpolating} approaches the topic in a more general setting via analyzing properties of Reproducing Kernel Hilbert Spaces (RKHS). Recent papers \cite{cao2021risk} and \cite{wang2020benign} are closely aligned with our work, but only consider the case when $r(\Sigma) =\Omega(n)$ and no regularization is present.  We summarize their key findings below.
\begin{table}[ht]
\begin{tabular}{|l|l|l|}
\hline
                  &  Wang and Thrampoulidis \cite{wang2020benign} & Cao et al.\cite{cao2021risk}  \\ \hline
SVP conditions  &   $\Tr(\Sigma)>C\left(\|\boldsymbol{\Sigma}\|_{F} \cdot n \sqrt{\log(n)}+\|\Sigma\|_{\infty} \cdot n^{3/2} \log(n)\right)$ & $\Tr({\Sigma}) \geq C \max \left\{n^{3 / 2}\|{\Sigma}\|_{\infty}, n\|{\Sigma}\|_{F}\right\}$ \\                                        &  \text{and} \quad $\Tr(\Sigma)>C_{1} n \sqrt{\log ( n) \theta^\top \Sigma \theta}$  &  \text{and} \quad $\Tr(\Sigma)>C_{1} n \sqrt{\log (n)\theta^\top \Sigma \theta} $ \\ \hline
Error bounds                        &     $\exp \left(\frac{-\left(\|{\theta}\|^{2}-\frac{C_{1} n \theta^\top \Sigma \theta}{\Tr(\Sigma)}-C_{2} \sqrt{\theta^\top \Sigma \theta}\right)^{2}}{C_{3} \max \left\{1, \frac{n^{2} \theta^\top \Sigma \theta}{\Tr(\Sigma)^{2}}\right\}\|\Sigma\|_{F}^{2}+C_{4} \theta^\top \Sigma \theta}\right)$ & $\exp \left(\frac{-C^{\prime} \|{\theta}\|^{4}}{\theta^\top \Sigma \theta+\|{\Sigma}\|_{F}^{2}/n+\|{\Sigma}\|_{\infty}^{2}}\right)$  \\ \hline
\end{tabular}
%\caption{Truth Tables and Accuracy Measures for each modeling library.}
\end{table}

%\begin{center}
%\scalebox{1}{\begin{tabular}{|l|l|l|l|l|}
%\hline
%Reference                                      & Type           & Noise & $p\gg n$& Asymptotic \\ \hline
%\cite{bartlett2020benign}     & Regression    & Isotropic & Yes    &  No       \\ \hline
%%\cite{wu2020optimal}          & Regression     & Anisotropic      & Yes &Yes       \\ \hline
%\cite{ndaoud2018sharp}        & Classification & Isotropic  & Yes   &No  \\ \hline
%\cite{wang2020benign}         & Classification & Anisotropic     & Yes    & No    \\ \hline
%\cite{wang2020efficient}      & Classification & Anisotropic      & No &   No   \\ \hline
%\cite{liang2021interpolating} & Classification & Isotropic      & Yes     &   No \\ \hline
%\cite{chen2021optimal}        & Classification & Anisotropic     & No    &   Yes \\ \hline
%\cite{chatterji2020finite}        & Classification & Anisotropic     & Yes    &   No \\ \hline
%\cite{mai2019high}        & Classification & Anisotropic     &  Yes    &   Yes \\ \hline
%\cite{cao2021risk}        & Classification & Anisotropic     & Yes    &   No \\ \hline
%\textbf{Our Work}                        & Classification & Anisotropic      & Yes & No       \\ \hline
%\end{tabular}}
%%\captionof{table}{Summary of the related work. }
%\end{center}

\subsection{Contribution} 
Our main contributions are summarized below and compared with the previous state of the art.
\begin{itemize}
\item First, we derive the minimax bounds for the generalization risk of clustering in the anisotropic sub-Gaussian model, and show that the averaging classifier, defined in Section \ref{sec:ave},  is adaptive and  minimax optimal.
\item We study the risk of the regularized least squares classifiers and show that, under mild assumptions, the interpolating solution $\hat{\eta}_{\text{OLS}}$ is minimax optimal, leading to a better bound than the previous works.  We also expose interesting cases where the interpolating solution may outperform regularized classifiers, under mild assumptions on the covariance structure of the noise.
\item Next, we show that the SVP phenomenon occurs under the mild conditions: $r(\Sigma) =\Omega(n\log(n))$ and $\Tr(\Sigma)=\Omega( n \sqrt{\log (n)\theta^\top \Sigma \theta} )$. As a consequence, we derive risk bounds for the hard-margin SVM classifier. The aforementioned conditions are strictly better than previously known ones, as the latter require that $r(\Sigma) =\Omega( n^{3/2}\log(n))$.
\item Finally, we propose the framework where the covariance of the noise can be corrupted, and show that interpolation leads to robust minimax optimal classifiers over a large class of signal vectors in this case, while both  the averaging and the LDA estimators fail. Hence not only interpolation is benign, but it can also be  optimal and robust.
\end{itemize}

\section{Minimax clustering: the supervised case}
\label{sec:ave}
In this section we consider the supervised clustering problem with Gaussian noise with unknown covariance $\Sigma$. Our upper bounds are still valid for the sub-Gaussian noise, but we only consider the Gaussian case and focus on other important aspects of the problem instead. We first state a lower bound. 
\begin{theorem}\label{thm:minimax} Let $\Delta,\lambda,r>0$. Then
\[
\underset{\hat{\eta}}{\inf}\underset{r(\Sigma^2)=r, \|\Sigma\|_\infty = \lambda }{\sup} \; \underset{\|\theta\|^2 \geq \Delta^2 \lambda }{\sup} \E(\R_{\Sigma}(\hat{\eta}))  \geq C \exp\left(-c\frac{\Delta^4}{\Delta^2 + \frac{r}{n}}\right),
\]
for some absolute constants $c,C>0$ and where the infimum is taken over all measurable classifiers.
\end{theorem}
Notice that the inequality is stated in terms of the Euclidean norm of $\theta$, and not the Mahalanobis norm that would lead to a different lower bound. This particular choice is motivated by the fact that we are seeking bounds that hold adaptively for large classes of $\Sigma$ which is typically unknown in practical applications. The proof of the lower bound is inspired by the argument in \cite{ndaoud2018sharp} that only holds for isotropic noise. As for the upper bound, we show that the ``averaging" linear classifier defined as
\[
\hat{\eta}_{\text{ave}}(y) := \sign\left( \left\langle \sum_{i=1}^n Y_i\eta_i,y \right\rangle \right),
\]
is minimax optimal. In fact, we will prove the following inequality.
\begin{theorem}\label{thm:upper} Let $\Delta>0$. For any covariance matrix $\Sigma$ we have, with probability at least $1-\delta - \ e^{-cn}$, that
\[
 \R(\hat{\eta}_{\text{ave}})  \leq C 
 \exp\left(-c\frac{\|\theta\|^4}{\theta^{\top}\Sigma \theta+ \frac{\Tr(\Sigma^2) + \|\Sigma\|^2_{\infty}\log(1/\delta)}{n} }\right).
\]
for some absolute constants $c,C>0$. Moreover,
\[
\underset{\|\theta\|^2 \geq \Delta^2 \|\Sigma\|_\infty }{\sup} \E(\R(\hat{\eta}_{\text{ave}})  )\leq C \exp\left(-c\frac{\Delta^4}{\Delta^2+ \frac{r(\Sigma^2)}{n}}\right).
\]
\end{theorem}
Theorem~\ref{thm:upper} provides a matching upper bound to Theorem~\ref{thm:minimax}. In particular, it implies that for consistency under the Euclidean norm, it suffices to assume that $\|\theta\|^2 \gg \|\Sigma\|_{\infty}( \sqrt{r(\Sigma^2)/n}+1)$. 
Among other conclusions, it yields that from a minimax perspective, the averaging classifier outperforms the LDA one. This fact is stated explicitly next.
\begin{prop}\label{prop:LDA} Let $p \geq n$. Then for any $\theta\in \mb R^p$
\[
 \E(\R(\hat{\eta}_{\text{LDA}})  )\geq C \exp\left(-c\frac{\|\theta\|^4_{\Sigma}}{\|\theta\|^2_{\Sigma}+ \frac{p}{n}}\right),
\]
for some $c,C>0$.
\end{prop}
Observe that for the worst case scenario, the vector $\theta$ is chosen so that $\Sigma \theta = \|\Sigma\|_\infty \theta$, whence $\R\left(\hat{\eta}_{\text{LDA}}\right)$ is sub-optimal.
This phenomenon is only possible in high dimensions. It is in fact easy to see that when $p \ll n$, LDA outperforms the averaging classifier for any given vector $\theta$ whenever consistency is possible. Indeed, in this case
\[
\frac{\|\theta\|^4_{\Sigma}}{\|\theta\|^2_{\Sigma}+ \frac{p}{n}}= \Omega\left( \|\theta\|^2_{\Sigma} \right)   .
\]
Moreover it is always true that
\[
 \|\theta\|^2_{\Sigma}  \geq  \frac{\|\theta\|^4}{\theta^\top \Sigma \theta}.
\]
Hence, whenever consistency is possible and when $p \ll n$, we have
\[
\frac{\|\theta\|^4_{\Sigma}}{\|\theta\|^2_{\Sigma}+ \frac{p}{n}}= \Omega\left(\frac{\|\theta\|^4}{\theta^\top \Sigma \theta +\frac{\Tr{\Sigma^2}}{n}} \right)
\]
Notice that the last statement is stronger than a minimax comparison.
%\end{comment}
%\remark{Minimax clustering: the unsupervised case.} We can actually extend the result of this section to the unsupervised case. Using the same procedure as in \cite{ndaoud2018sharp} we claim the existence of a polynomial time method that is minimax optimal for clustering.\\

\section{ Interpolation vs Regularization for  sub-Gaussian mixtures}

In this section, we study the risk of the regularized OLS estimators. While it is more common to study SVM for classification, recent works \cite{ardeshir2021support,hsu2020proliferation} have shown that in high dimensions (specifically, $p=\Omega(n\log{n})$), SVM and OLS solutions coincide under mild conditions. This phenomenon is known in the literature as ``proliferation of support vectors". One of its implications is that in high dimensions, it is sufficient to study the properties of the least squares estimator and then demonstrate that it coincides with the hard-margin SVM. For the rest of this section, our goal is to study the risk of the family of supervised estimators defined as solutions to the problem
\[
\min_{\bar{\theta} \in \mb R^{p}} \frac{1}{n}\sum_{i=1}^n (\eta_{i} - \langle Y_i , \bar{\theta}\rangle)^2 + \lambda \|\bar{\theta}\|^2.
\]
Observe that the case $\lambda\to 0$ and $p \geq n$ leads to interpolation, specifically to the minimum $\ell_2$-norm interpolating solution. %(cf \cite{bartlett2020benign} for interpolation in regression and \cite{wang2020benign,cao2021risk} for interpolation in clustering of Gaussian mixtures, cf. also \cite{liang2021interpolating}). 
For each $\lambda > 0$, the corresponding estimator $\hat{\theta}_\lambda$ is proportional to
\[
\hat{\theta}_\lambda =  \frac{1}{n}\left(\lambda I_p + \frac{1}{n}YY^\top \right)^{-1}Y\eta =\frac{1}{n}Y \left(\lambda I_n + \frac{1}{n}Y^\top Y \right)^{-1}\eta .
\]
Each estimator $\hat{\theta}_\lambda$ leads to a linear classifier defined via
\[
\hat{\eta}_{\lambda}(\cdot) = \sign\left( \left\langle \hat{\theta}_\lambda,\cdot \right\rangle \right).
\]
In what follows, we will derive upper bounds for the risk $\R(\hat{\eta}_{\lambda})$. We will also provide sufficient conditions for the matrix $\Sigma$ ensuring that the interpolating classifier (corresponding to $\lambda = 0$) achieves at least the same performance as the oracle (or the averaging classifier that is minimax optimal). 

Notice that as $\lambda$ goes to $\infty$, we recover the averaging classifier. We emphasize here that if we replace the regularization term $\|\theta\|^2$ by $\|\theta\|_{\Sigma}^2$ (not adaptive to $\Sigma$), then our claims may no longer be true.

As a side note, we suspect that the excess risk of estimating $\theta$ itself gets smaller for large values of $\lambda$ although this is not necessarily the case for the classification risk. This suggests that the excess risk of estimation may not be the right metric to evaluate classification performance.

In order to state our main result, we will need the following quantity. For a given covariance matrix $\Sigma$, we define the function $k_{\Sigma}^{*}$ via
\[
k_{\Sigma}^{*}(\lambda) = \min \left\{k \geq 0, r_{k}(\Sigma) + \frac{n\lambda}{\lambda_{k+1}} \geq  C_1 n \right\},
\]
for some absolute constant $C_1 > 1$ (that can be chosen to be large enough) and $k_{\Sigma}^{*}(\lambda) := p+1$ if the above set if empty. 
The reader may observe that $k_{\Sigma}^*(\cdot)$ is decreasing with $\lambda$ and that $k_{\Sigma}^*(\lambda) = 0$ if $r(\Sigma) \geq C_1 n$. In what follows, we will require $k^*_\Sigma(\lambda)$ to be smaller than $n/2$. For $\lambda = 0$, this means that we are not allowing more than a fraction of all eigenvalues to be much larger than the remainder of the spectrum. Such conditions encompass many covariance matrices of interest, and in particular cover the case $\Sigma = \mathbf{I}_p + R$ where $R$ can be viewed a low rank perturbation/corruption component. In what follows, $\pi_{k^*}$ will stand for $\pi_{k^*}(\Sigma)$.
\begin{theorem}
\label{thm:interpolation1} 
Let $\Delta>0, \lambda \geq 0$. Assume that $k_{\Sigma}^{*}(\lambda) \leq n/2$ and that $\|\pi_{k^*}\theta\|^2 \leq \|\theta\|^2/5 $. Then for some constants $c,C>0$ we have with probability $1-\delta - e^{-cn}$ that 
\[
\R(\hat{\eta}_{\lambda})  \leq C\exp\left(-c\frac{\|\theta\|^4}{\theta^\top  \Sigma \theta(1+k^{*}) +  \frac{k^* \lambda^2_{k^*} + \sum_{i>k^*}\lambda^2_i(\Sigma)}{n}+\frac{(k^{*}\lambda_{k^*}^2+\lambda_{k^*+1}^2)\log(1/\delta)}{n}}\right),
\]
where $k^* = k_{\Sigma}^*(\lambda)$.
\end{theorem}

The condition  $\|\pi_{k^*}\theta\|^2 \leq \|\theta\|^2/5 $ connecting $\theta$ with $\Sigma$ can be understood as follows. One may think of the $k^{*}$  eigenvectors corresponding to largest eigenvalues of $\Sigma$ as ``outliers'', or directions affecting dramatically the rest of the spectrum. Our condition prohibits $\theta$ from concentrating too much of its mass in the subspace spanned by the latter eigenvectors. For instance, if $\theta/\|\theta\|$ is a random vector that is well spread out on the sphere (spherically distributed) then $\|\pi_{k^*}\theta\|\approx \frac{k^{*}}{p}\|\theta\| \ll \|\theta\|$. The same remark holds also if we fix $\theta$ and think  of the range of $\pi_{k_\ast}$ as being a random subspace making $\pi_{k_\ast} \theta$ a random vector.
%as long as $\theta^\top \Sigma \theta /\|\theta\|^2 \leq \lambda_{k^{*}+1}$ (which holds, in particular, if $\theta$ is an eigenvector corresponding to any eigenvalue $\lambda_i$ as long as $i>k^*$) then $\theta^\top \Sigma \theta /\|\theta\|^2 \leq C( \sum_{i>k^*}\lambda_i/n + \lambda) $ holds.
Hence the latter condition means simply that the vector $\theta$ is only allowed to be aligned with the ``clean'' part of the covariance $\Sigma$. 

When $k^{*}=0$ (or equivalently $r(\Sigma) \geq C_1 n$), we recover the bound obtained in \cite{cao2021risk} by taking $\delta = e^{-cn}$. Notice that in this case the alignment condition $\|\pi_{k^*}\theta\|^2 \leq \|\theta\|^2/5 $  is always satisfied as $\pi_{k^*} = 0_p$. Our result is stronger since we show that the bound holds with probability $1-e^{-cn}$ while in \cite{cao2021risk} authors only prove that the same bound holds with probability $1-1/n$. Moreover, under the mild condition $r(\Sigma^2) \geq \log(n)$ and by taking $\delta=1/n$, we show that 
\[
\R(\hat{\eta}_{\lambda})  \leq C\exp\left(-c\frac{\|\theta\|^4}{\theta^\top  \Sigma \theta+  \frac{\Tr(\Sigma^2)}{n}}\right)
\]
with probability at least $1-1/n$. Therefore, taking the limit as $\lambda \to 0$ leads to the same bound as the averaging oracle in this case. 

In the case of moderate values of $k^*$, it is worth noticing  that $k^*(\cdot)$ is a decreasing function of $\lambda$. As a consequence, $k^* \lambda_{k*}^2 + \sum_{i>k^*} \lambda_{i}^{2}$ is an increasing function of $\lambda$ such that
\[
k^* \lambda_{k^*}^2 + \sum_{i>k^*} \lambda_{i}^{2} \leq \Tr{(\Sigma^2)}.
\]
The latter quantity could be seen as a truncated trace of $\Sigma^2$, where truncation is applied to the largest eigenvalues of $\Sigma$.
Unlike the previous works \cite{cao2021risk,wang2020benign}, our result is more general since we allow $k^*$ to be non-zero.  The bound in Theorem \ref{thm:interpolation1} in particular gets smaller as $\lambda$ goes to $0$, which suggests that interpolation may outperform regularization in some cases, especially in the case where a finite number of eigenvalues are much larger than the rest of the spectrum. 

In the case where $k^*$ can be large we get the following result, without any further assumptions on $\Sigma$. Let us define
\[
\mathcal{C}_{k^*}(\Sigma):= \{\theta \in \mathbb{R}^p,\quad \|\pi_{k^*}\theta\| \leq \|\theta\|/\sqrt{5}\}.
\]
\begin{prop}\label{cor:interpolation1} Let $\Delta>0, \lambda \geq 0$.  Assume that $k_{\Sigma}^{*}(\lambda) \leq n/2$. Then for some absolute constants $c,C>0$ 
\[
\underset{\substack{\|\theta\|^2 \geq \Delta^2 \|\Sigma\|_\infty \\ \theta \in \mathcal{C}_{k^*}(\Sigma)
 }}{\sup} \E(\R(\hat{\eta}_{\lambda}) ) \leq C\exp\left(-c\frac{\Delta^4}{\Delta^2 + \frac{r(\Sigma^2)}{n}}\right) + e^{-cn}.
\]
\end{prop}
This result suggests that under a mild condition on the covariance of the noise, not only interpolation is benign but it is also minimax optimal on the set $\mathcal{C}_{k^*}(\Sigma)$. This also means that interpolation is better for classification than for regression since it does not suffer from a bias term which often leads to bad worst-case performance (\cite{bartlett2020benign}).

Recall that all our results hold under the condition $k_{\Sigma}^{*}(\lambda) \leq n/2$. We may wonder here what happens if  $k^{*}$ is much larger than $n$, and numerical experiments suggest that interpolation indeed behaves poorly in this case.

\section{Proliferation of support vectors in high dimensions under the sub-Gaussian mixture model}

In this section, we provide sufficient conditions for proliferation of support vectors. Based on the results in  \cite{hsu2020proliferation}, $\hat{\theta}_{\text{SVM}}$ and $\hat{\theta}_{\text{OLS}}$, as defined in \eqref{eq:SVM}-\eqref{eq:OLS}, coincide  if and only if 
\[
\forall i=1,\dots,n \quad \eta_i e_i^\top(Y^{\top}Y)^{-1}\eta  >0,
\]
where $(e_i)_{i=1,\dots,n}$ is the Euclidean canonical basis.  In the remainder of the section we denote $k^{*}:= k^{*}_{\Sigma}(0)$. The main result is stated next.
\begin{theorem}
\label{thm:prolif}
Assume that $k^{*}\log^2(n) \leq C n$, $\sum_{i>k^*} \lambda_{i}^{2}n\log(n) \leq C (\sum_{i>k^*} \lambda_{i})^2$ and $ \sqrt{\theta^\top \Sigma \theta (1+k^*)\log(n)} \leq C \sum_{i>k^*}\lambda_i/n$ for some absolute constant $C>0$. Then with probability at least $1-1/n$,
\[
\hat{\theta}_{SVM} = \hat{\theta}_{\text{OLS}}.
\]
\end{theorem}
As a consequence, $\hat{\eta}_{\text{SVM}}$ attains the same performance as $\hat{\eta}_0$ under the conditions of Theorem \ref{thm:prolif}.

When $k^{*} = 0$ (i.e. $r(\Sigma) = \Omega(n)$), the sufficient conditions read as
\begin{itemize}
    \item $ \sqrt{\theta^\top \Sigma \theta \log(n)} \leq C\Tr(\Sigma)/n$;
\item $\Tr(\Sigma^2) n\log(n) \leq C(\Tr(\Sigma))^2$.
\end{itemize}
 The first condition (that is signal-dependent) is also required in prior works \cite{cao2021risk,wang2020benign}. As for the ``dimension-dependent'' second condition, it is much milder than the one proposed in both previous papers on the topic. To compare these results, consider the case $\Sigma = \mathbf{I}_p$, where our condition reads as $p = \Omega(n\log(n))$ while the earlier analogues require $p = \Omega(n^{3/2}\log(n))$. Our result also suggests that $r(\Sigma) = \Omega(n\log(n))$ suffices for proliferation under the sub-Gaussian mixture model, which confirms the general conjecture stated in \cite{hsu2020proliferation}.

\section{Application to robust supervised clustering}
\label{sec:robust}
In this section, we present the framework for robust clustering in the sub-Gaussian mixture model. In the rest of this section we consider the case of Gaussian noise and identity covariance $\Sigma = \mathbf{I}_p$.  We will assume that the training set has a different covariance than the covariance of the new observation to be classified, due to the action of a malicious adversary.  More precisely, given the vector of observations $Y$, the adversary can corrupt the sample as follows: she chooses up to $r \leq n/4$ eigenvalues of the covariance matrix $\Sigma$ and positive scalars $O_1,\dots,O_r$, and then adds i.i.d. random noise to $Y$ such that the new observations become
\[
\tilde{Y}_i = Y_i +O^{1/2}\epsilon_{i}  = \eta_i\theta  + W_i + O^{1/2} \epsilon_{i}, \quad  \forall i=1,\dots,n, 
 \]
 where $\epsilon_{i}$ are i.i.d. standard normal vectors  that are also independent from $Y$, and $O = \sum_{i \in R} O_ie_ie_i^{\top}$ where $(e_1,\dots,e_p)$ is the canonical Euclidean basis of $\mathbb{R}^p$ and  $R$ is the set of indices corresponding to the corrupted eigenvalues. Observe that the covariance of the noise is now given by $\mathbf{I}_p + O$.  In what follows,  $\pi_{r}$ denotes the projection $\pi_{r}( O):= \sum_{i \in R}e_ie_i^\top$. Our goal is to show that under minimal assumptions, the interpolating estimator is still minimax optimal while both the averaging estimator and the $\text{LDA}$ one fail to perform well.

\begin{theorem}\label{thm:robust}
Assume that $r  \leq n/4$ and that $\Delta^2 \geq  p/n$. Then
\[
\underset{\substack{\|\theta\| \geq \Delta  \\ \theta \in \mathcal{C}_{r}(\mathbf{I}_p + O)
 }}{\sup} \E(\R(\hat{\eta}_{0}) ) \leq C\exp\left(-c\frac{\Delta^4}{\Delta^2 + \frac{p}{n}}\right) + e^{-cn}.
\]
% and $r \Tr(\Sigma)^2 \leq \Tr(\Sigma^2)n^2$.

%\begin{itemize}
%\item 
%If $\|\pi_{r}\theta\|^2 \leq \|\theta\|^2/5 $, then with probability $1-\frac{1}{n}$ we have
%\[
%\R_{\Sigma}(\hat{\eta}_{0})  \leq C\exp\left(-c\frac{\|\theta\|^4}{\|\theta\|^2(1+r)+  \frac{p}{n}}\right).
%\]
%If moreover conditions of Theorem \ref{thm:prolif} hold,  then with probability $1-\frac{1}{n}$ we get
%\[
%\R_{\Sigma}(\hat{\eta}_{\text{SVM}})  \leq C\exp\left(-c\frac{\|\theta\|^4}{\theta^\top \Sigma \theta(1+r + k^{*})+  \frac{\Tr(\Sigma^2)}{n}}\right).
%\]
%\item From a minimax perspective, we get moreover that
%\[
%\underset{\substack{\|\theta\| \geq \Delta  \\ \theta \in \mathcal{C}_{r}(O)
% }}{\sup} \E(\R(\hat{\eta}_{0}) ) \leq C\exp\left(-c\frac{\Delta^4}{\Delta^2 + \frac{p}{n}}\right) + e^{-cn}.
%\]
%for some $c,C>0$.
%\end{itemize}
\end{theorem}
One implication of this theorem is the fact that $\hat{\eta}_0$ is minimax optimal on the set $\mathcal{C}_{r}(\mathbf{I}_p + O)$  and robust with respect to the corruption $O$, for moderate values of $p$.  When $r \leq n/4$ and the direction of $\theta$ is not too closely aligned with the eigenvectors corresponding to the corrupted part of the spectrum, then $\hat{\eta}_0$ mimicks the performance of the averaging estimator in the absence of outliers. %Moreover  $\hat{\eta}_{\text{SVM}}$ achieves the same rate as long as $\|\theta\| = \mathcal{O}( p/ (n\sqrt{\log(n)})$ and  $n\log(n) = \mathcal{O}(p)$.
In order to compare the bound stated above to estimates available for other estimators, we rely on the next proposition.
\begin{prop}\label{prop:compare}
Assume that the noise is Gaussian and that $O$ satisfies $O = n \sum_{i=1}^r e_ie_i^\top$ where  $r=n/4$. Then for any $\theta$ such that $\|\theta\|^2 \leq \sqrt{n}$, 
\[
\E(\R_{\Sigma}(\hat{\eta}_{\text{LDA}})) \wedge \E(\R_{\Sigma}(\hat{\eta}_{ave}) ) \geq C
\]
for some absolute constant $C>0$.
\end{prop}
In summary, when $r  \leq n/4$, there exists a regime (where $p/n\leq\|\theta\|^2 \leq \sqrt{n}$) such that interpolation $\hat{\eta}_{0}$ is minimax optimal over a large class of vectors $\theta$, under the contamination model we presented, while both $\hat{\eta}_{ave}$ and $\hat{\eta}_{\text{LDA}}$ can fail with constant non-zero probability.

\begin{comment}
\section{Interpolation is all you need}
We now focus on the case $\lambda = 0$. Previous we have that \[
\hat{\theta}=\left(\frac{YY^\top }{n}\right)^{-1}\frac{Y\eta}{n}
\] 

\begin{theorem}\label{thm:interpolation2} Let $\Delta>0$. Assume that $k_{\Sigma}^{*}(0) \leq n/2$ and $\theta^\top \Sigma \theta \leq \lambda_{k_{\Sigma}^*(0)} \|\theta\|^2$. Then for some constants $c,C>0$ we have that 
\[
R(\hat{\eta}_{0})  \leq C\exp\left(-c\frac{(\theta^\top  \Sigma^{-1}\theta)^2}{\frac{p}{n}}\wedge \frac{\|\theta\|^4}{\theta^\top \Sigma \theta}\wedge \frac{p}{n}\right).
\]
\end{theorem}
The result implies that there exist regimes where interpolation mimicks the LDA estimators, while in the other regimes it mimicks the mean estimator based on whether the center vector $\theta$ is aligned with the top eigenvectors of $\Sigma$ or not. As a consequence we derive the following corollary that holds under a mild assumption on the condition number of $\Sigma$.
\begin{corollary}\label{cor:interpolation2}
Assume that $\kappa(\Sigma) \leq \sqrt{p/n}$. Then interpolation achieves consistent clustering $(R(\hat{\eta}) \to 0)$ if
\[
\| \theta \|^2 \gg \sqrt{\frac{\Tr(\Sigma^2)}{n}} \text{ or } \theta^\top \Sigma^{-1} \theta  \gg \sqrt{\frac{p}{n}}.
\]
\end{corollary}

Even interpolation is not fully optimal, it takes the best of the empirical
mean oracle and the LDA estimator without being explicit about which method to
use when \(p\gg n\).
\end{comment}
\section{Numerical experiments} 
The goal of this section is to compare the performance of several estimators $\hat{\eta}_{\lambda}$ for different values of $\lambda$. The case $\lambda = 0$ corresponds to interpolation, while $\lambda = \infty$ recovers the averaging classifier $\hat{\eta}_{\text{ave}}$. In all our simulations we will only consider Gaussian noise and $\theta/\|\theta\|$ spherically distributed. 

\begin{figure}[ht!]
\minipage{0.5\textwidth}
  \includegraphics[width=\linewidth]{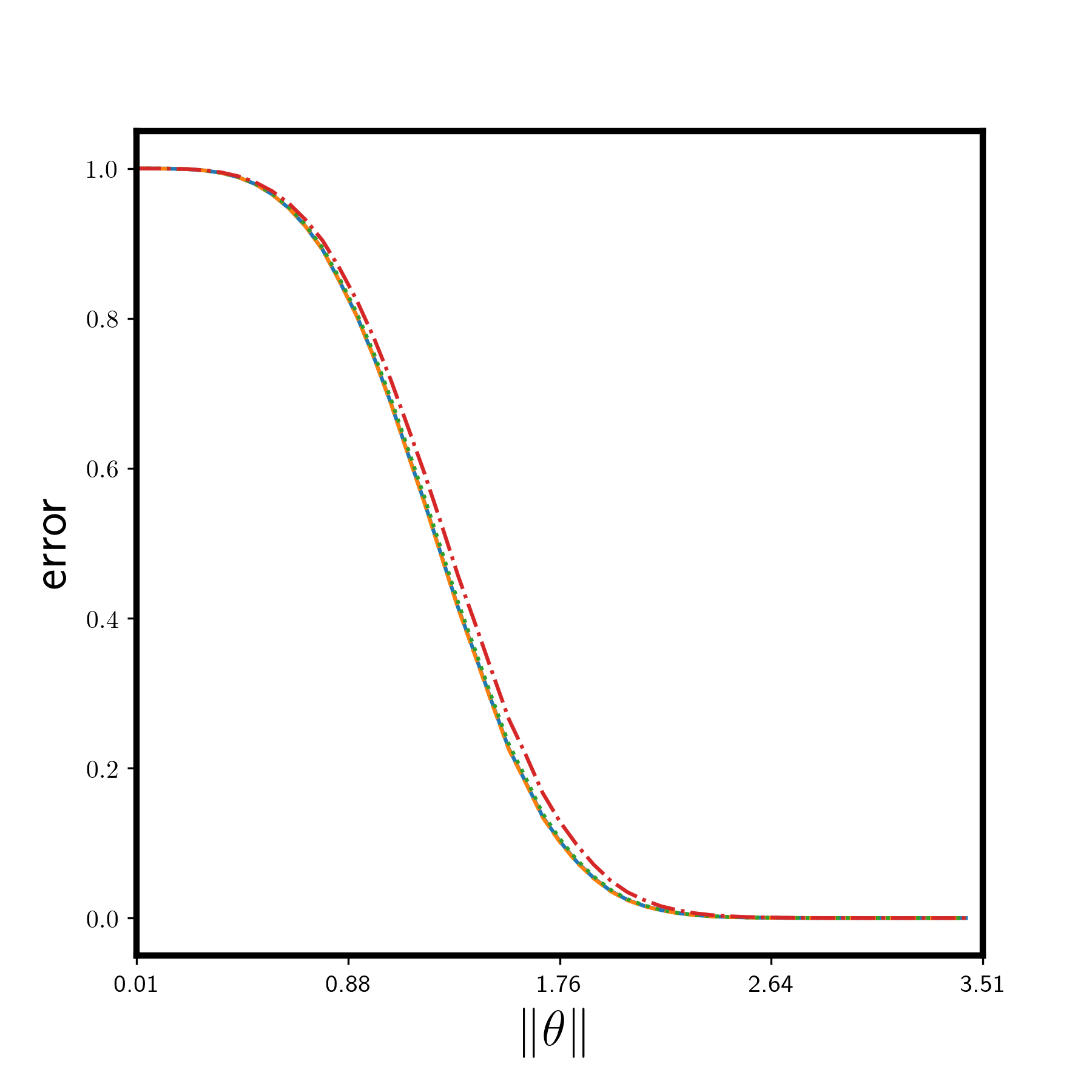}
  %\caption{$k^* = 0$}\label{fig:awesome_image1}
\endminipage\hfill
\minipage{0.5\textwidth}
  \includegraphics[width=\linewidth]{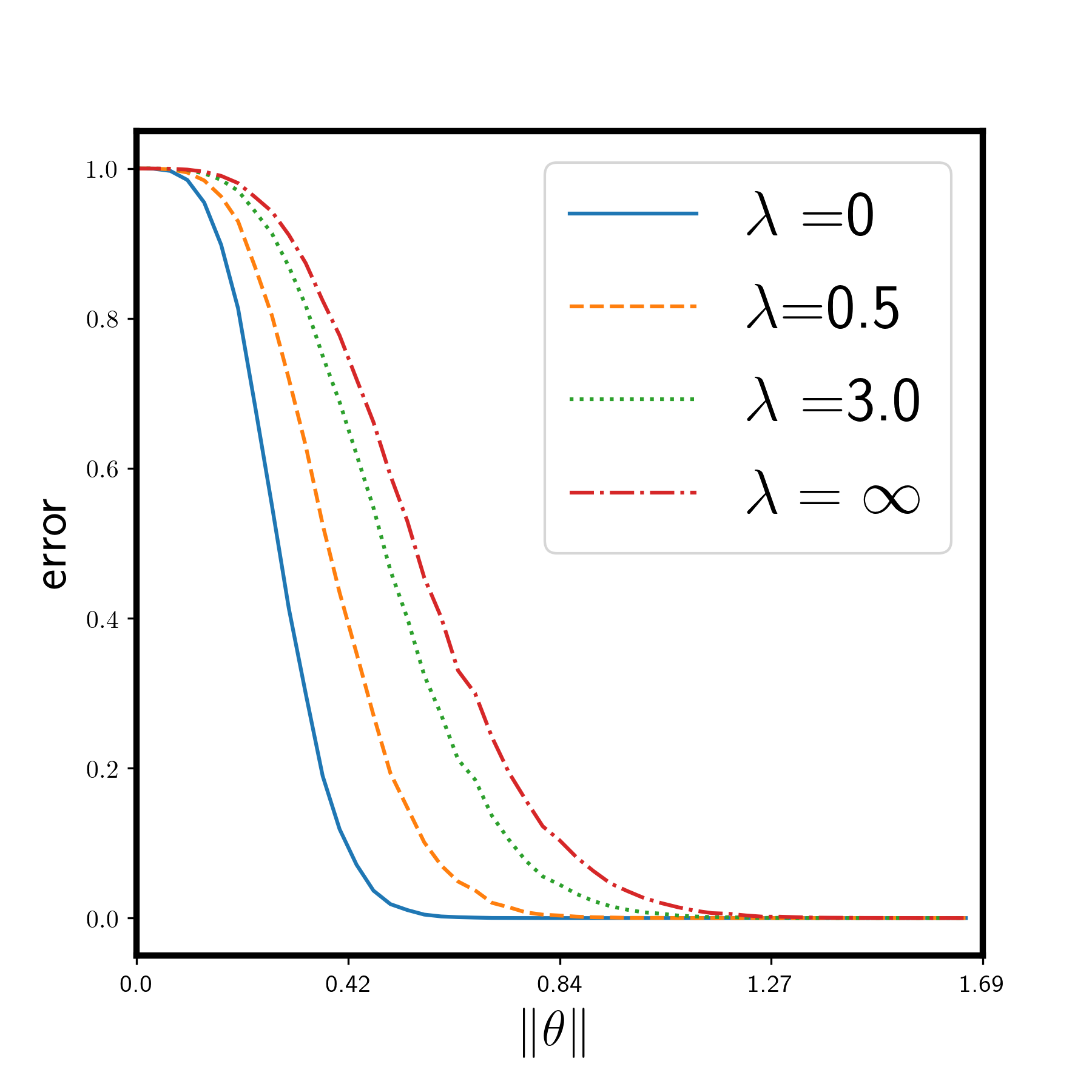}
  %\caption{$k^* = 3$}\label{fig:awesome_image2}
\endminipage
\caption{Comparison of the generalization error of four classifiers. The left plot corresponds to the case of large effective rank while the right one corresponds to medium effective rank.}
\label{fig:1}
\end{figure}
Our simulation setup is defined as follows. We choose $p = 500$ and $n=30$ to allow for overparametrization. We compare 4 estimators: the interpolating classifier $\hat{\eta}_0$,  classifiers $\hat{\eta}_\lambda$ for $\lambda = 0.5$ and $\lambda =3$, as well as the averaging estimator $\hat{\eta}_{\text{ave}}$. 
For each value of $\|\theta\|$, simulation was repeated $1000$ times; finally, we plot the empirical generalization error. 

For our first experiment (Figure \ref{fig:1}), we compare performance of the four classifiers in two cases:
\begin{itemize}
\item The case of large effective rank where we choose $\Sigma$ to be a diagonal matrix with $\lambda_i = (p-i+1)/p$ for all $i=1,\dots,p$. This case corresponds to $k^{*}(0) = 0$;
 \item The case of medium effective rank where we choose $\Sigma$ to be a diagonal matrix with $\lambda_1=\lambda_2=\lambda_3 = 1$ and $\lambda_4=\dots=\lambda_p= 0.01$. This case corresponds to $k^{*}(0) = 3$.
\end{itemize}

Consistent with the theoretical predictions, our simulations suggest that all classifiers have similar performance in the regime of large effective rank. Interestingly, interpolation seems to perform best when the effective rank is smaller than $n$. This confirms our observation that interpolation can be superior to regularization in some cases.

As for our second experiment (Figure \ref{fig:2}), we choose $\Sigma =  \mathbf{I}_p$ and we corrupt the training sample, as explained in Section \ref{sec:robust}, by setting randomly $n/2$ diagonal entries of the covariance to $1000$. Remember that this modification only impacts the training sample while the test sample has isotropic noise. 

\begin{figure}[ht]
\includegraphics[width=7cm]{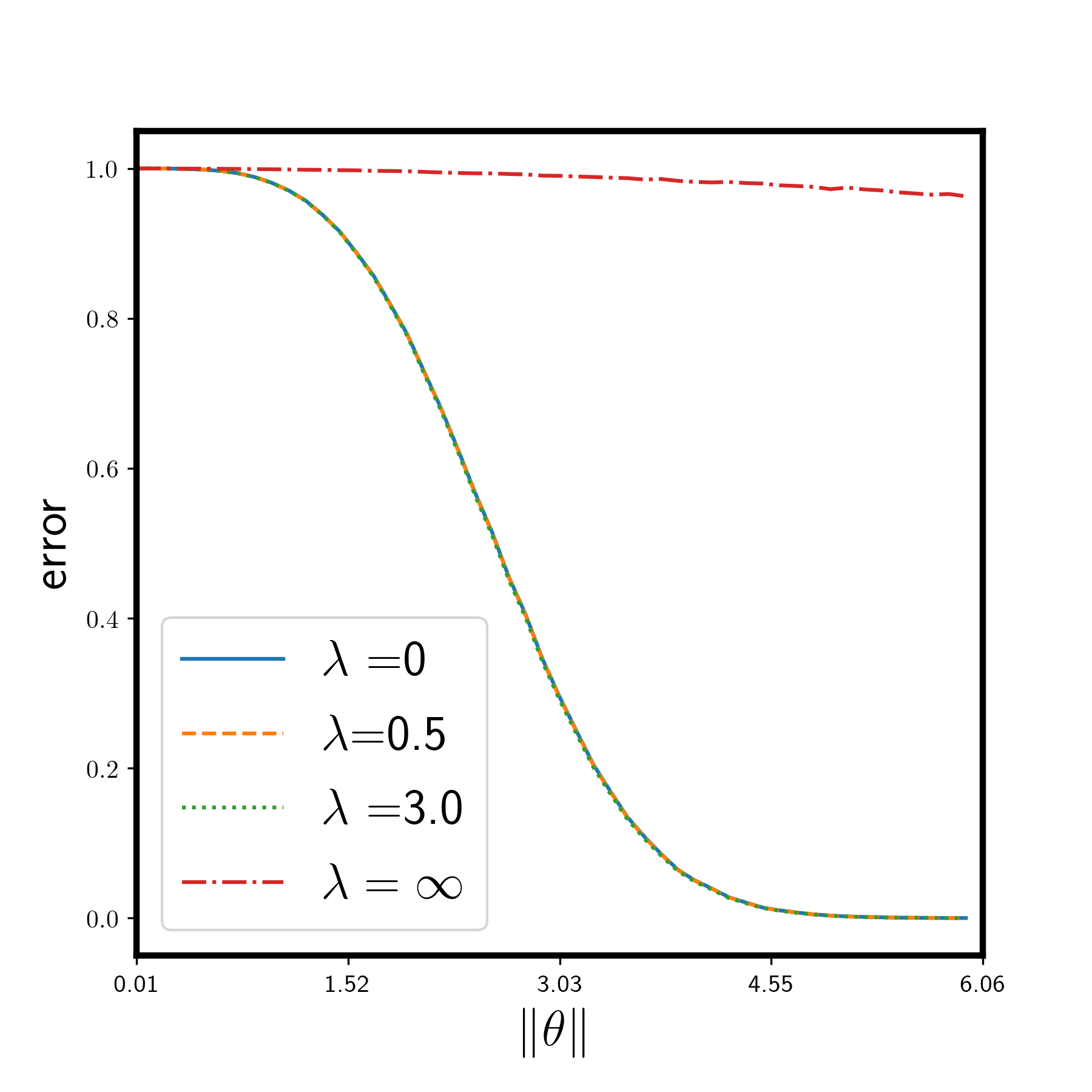}
\caption{Comparison of the performance of four classifiers under corruption.}
\label{fig:2}
\end{figure}
In this case, the averaging classifier fails to predict new labels correctly while the interpolating classifier is still able to classify well despite the corruption. Observe also that regularized classifiers (corresponding to the small regularization parameter) perform similar to the interpolating classifier. This occurs simply due to the fact that $k^*_{\tilde{\Sigma}}(\lambda) = k^*_{\tilde{\Sigma}}(0)$ for all values of $\lambda$ much smaller than the magnitude of the corruption, where $\tilde{\Sigma}$ is the corrupted covariance matrix. We conclude by reaffirming the claim that interpolation is not only harmless in high dimensions, but can outperform regularization and in some cases exhibits an unexpected feature of robustness.

\bibliographystyle{imsart-nameyear}
%\bibliography{ref.bib}

\appendix
%\newpage

\section{Proofs of minimax clustering}
\subsection{Proof of Theorem~\ref{thm:minimax}}
Fix $\lambda,r>0$ and let $\Sigma$ be a diagonal PSD matrix such that $r(\Sigma^2) = r$ and $\|\Sigma\|_\infty = \lambda$. Then
\[
\underset{\hat{\eta}}{\inf}\underset{r(\Sigma_o^2)=r, \|\Sigma_o\|_\infty = \lambda }{\sup} \; \underset{\|\theta\|^2 \geq \Delta^2 \lambda }{\sup} R_{\Sigma_o}(\hat{\eta})  \geq \underset{\hat{\eta}}{\inf} \underset{\|\theta\|^2 \geq \Delta^2 \lambda }{\sup} R_{\Sigma}(\hat{\eta}).
\]
We can simply focus on showing the result for the above diagonal matrix $\Sigma$. We will use the fact that
\[
2\;\underset{\hat{\eta}}{\inf} \underset{\|\theta\|^2 \geq \Delta^2 \lambda }{\sup} \mathbf{E}(\mathcal{R}_{\Sigma}(\hat{\eta})) \geq \underset{\tilde{\eta}}{\inf}\mathbb{E}_{\pi}\mathbf{E}_{(\theta,\eta,\eta_{n+1})}|\tilde{\eta}((Y,\eta);Y_{n+1})- \eta_{n+1}|
\]
for any prior $\pi$ on $(\theta,\eta,\eta_{n+1})$ such that $\|\theta\|^2 \geq \Delta^2 \lambda$. The quantity $\mathbf{E}_{(\theta,\eta,\eta_{n+1})}(\cdot)$ stands for the expectation corresponding to $(Y;Y_{n+1})$ following model \eqref{eq:model}. The proof is decomposed in two steps.
\begin{itemize}
\item \textbf{A dimension-independent lower bound:}
 Let $\bar{\theta}$ be a fixed vector in $\mathbb{R}^{p}$ such that $\|\bar{\theta}\|^2=\Delta^2 \lambda$.
We place independent Rademacher priors $\pi_i$ on each $\eta_i$ for $i=1,\dots,n+1$. It follows that
\begin{equation}\label{eq:proof:triv:1}
\underset{\tilde{\eta}}{\inf}\mathbb{E}_{\pi}\mathbf{E}_{(\bar{\theta},\eta,\eta_{n+1})}|\tilde{\eta}((Y,\eta);Y_{n+1})- \eta_{n+1}| \geq \underset{\bar{\eta}}{\inf}\mathbb{E}_{\pi_{n+1}}\mathbf{E}_{(\bar{\theta},\eta_{n+1})}|\bar{\eta}(Y_{n+1}) - \eta_{n+1}|, 
\end{equation}
where $\bar{\eta}(Y_{n+1}) = \mathbf{E}(\tilde{\eta}((Y,\eta);Y_{n+1})|Y_{n+1}) \in [-1,1]$ as we average over $(Y,\eta)$.
The last inequality holds in view of Jensen's inequality and the independence between $(Y,\eta)$ and $\eta_{n+1}$. We define, for $\epsilon \in \{-1,1\}$, $\tilde{f}_{\epsilon}(\cdot)$ the density of the observation $Y_{n+1}$ conditionally on the value of $
\eta_{n+1}=\epsilon$. Now, using Neyman-Pearson lemma and the explicit form of $\tilde{f}_{\epsilon}$, we get that the selector $\eta^{*}$ given by
$$
\eta^{*} = \sign\left( \bar{\theta}^\top \Sigma^{-1} Y_{n+1}\right),
$$
is optimal as it achieves the minimum of the RHS of \eqref{eq:proof:triv:1}. 

To show that, we remind the reader that the distribution of $Y_{n+1}=\eta_{n+1}\bar{\theta}+W_{n+1}$, conditionally on $\eta_{n+1}$, is given by $\mathcal{N}(\eta_{n+1}\bar{\theta},\Sigma)$. Hence $$\tilde{f}_\epsilon(Y_{n+1})=(2\pi)^{-p/2}|\Sigma|^{-1/2}e^{-\frac{1}{2}(Y_{n+1}-\epsilon \bar{\theta})^\top \Sigma^{-1}(Y_{n+1}-\epsilon\bar{\theta})}.$$
It follows that
$$
\begin{aligned}
\frac{\tilde{f}_{1}(Y_{n+1})}{\tilde{f}_{-1}(Y_{n+1})}&=\frac{(2\pi)^{-p/2}|\Sigma|^{-1/2}e^{-\frac{1}{2}(Y_{n+1}- \bar{\theta})^\top \Sigma^{-1}(Y_{n+1}-\bar{\theta})}}{(2\pi)^{-p/2}|\Sigma|^{-1/2}e^{-\frac{1}{2}(Y_{n+1}+\bar{\theta})^\top \Sigma^{-1}(Y_{n+1}+\bar{\theta})}}\\
&=e^{2\bar{\theta}^\top \Sigma^{-1}Y_{n+1}}\\
\end{aligned}
$$
By Neyman-Pearson lemma, we can now conclude that
$$
\eta^{*} = \sign\left( \bar{\theta}^\top \Sigma^{-1} Y_{n+1}\right).
$$

Plugging this value in \eqref{eq:proof:triv:1}, we get further that 
$$
\underset{\bar{\eta}}{\inf}\mathbf{E}_{\pi}|\bar{\eta}(Y_{n+1}) - \eta_{n+1}| = 2\mathbf{E}\R_{\Sigma}(\eta^*).
$$
It is  straightforward to see that 
\[\mathbf{E}(\R_{\Sigma}(\eta^*))=\Phi^c\left(\sqrt{\bar{\theta}^\top \Sigma^{-1} \bar{\theta}}\right) \geq C e^{-c \bar{\theta}^\top \Sigma^{-1} \bar{\theta}},\]
for some $c,C>0$ where $\Phi(\cdot)$ is the standard normal cumulative function.
The last inequality holds for any $\bar{\theta}$ as long as $\|\bar{\theta}\|^2=\Delta^2 \lambda$. The worst case is reached for $\bar{\theta}$ being co-linear with the top eigenvector of $\Sigma$ since $\bar{\theta}^\top \Sigma^{-1} \bar{\theta} = \Delta^2$. Hence we get the lower bound
$$
\underset{\hat{\eta}}{\inf} \underset{\|\theta\|^2 \geq \Delta^2 \lambda }{\sup} \mathbf{E}(\R_{\Sigma}(\hat{\eta})) \geq Ce^{-c \Delta^2}.
$$
In order to conclude we only need to derive the other lower bound
$$
\underset{\hat{\eta}}{\inf} \underset{\|\theta\|^2 \geq \Delta^2 \lambda }{\sup} \mathbf{E}(\R_{\Sigma}(\hat{\eta})) \geq  C e^{- c n \Delta^4/r}.
$$
For the rest of the proof we only focus on the case where $ 100 \leq \Delta^2 \leq 10r/n$, otherwise the  dimension independent lower bound dominates.

\item \textbf{A dimension-dependent lower bound:}

Since $\Sigma$ is diagonal, we write $\Sigma=\diag(d_1,\dots,d_p)$ where $\lambda = d_1 \geq d_2 \geq \dots \geq d_p > 0$.
According to Theorem $1$ in \cite{ndaoud2018sharp} we have
\begin{multline*}  
2\;\underset{\hat{\eta}}{\inf} \underset{\|\theta\|^2 \geq \Delta^2 \lambda }{\sup} \mathbf{E}(\R_{\Sigma}(\hat{\eta})) 
\\
\geq \underset{T\in[-1,1]}{\inf}\mathbb{E}_{\pi} \mathbf{E}_{(\theta,\eta,\eta_{n+1})}|T((Y,\eta);Y_{n+1}) - \eta_{n+1}|  - 2\pi(\|\theta\|^2 \leq \Delta^2 \lambda) ,
\end{multline*}
for any prior $\pi$ on $(\theta,\eta,\eta_{n+1})$. The second term in the above lower bound accounts for the constraint on $\theta$. 
In what follows, we fix $\eta$ and choose $\pi^D$ to be a product prior on  ($\theta$,$\eta_{n+1}$) such that $\eta_{n+1}$ is a Rademacher random variable and $\theta$ is an independent random vector such that $\theta \sim\mathcal{N}(0,D)$, where $D$ is a diagonal matrix such that $D_{jj} = 2\Delta^2 \lambda \frac{d_j^2}{\sum_{i=1}^p d_i^2}$. Using the Hanson-Wright inequality \cite{rudelson2013hansonwright}, we deduce that
\[
\pi^D(\|\theta\|^2 \leq \Delta^2 \lambda) \leq Ce^{-c\; r},
\]
for some $c,C>0$. Since $\Delta^2 \leq 10 r/n$, we only need to show that, for $n$ large enough, we have
\[
\underset{T\in[-1,1]}{\inf}\mathbb{E}_{\pi} \mathbf{E}_{(\theta,\eta,\eta_{n+1})}|T((Y,\eta);Y_{n+1}) - \eta_{n+1}| \geq Ce^{-cn\Delta^4/r},
\]
for some $c,C>0$.
We define, for $\epsilon \in \{-1,1\}$, $\tilde{f}_{\epsilon}$ to be the density of the observation $(Y;Y_{n+1}) \in \mathbb{R}^{p \times (n+1)}$ given $\eta_{n+1}=\epsilon$. Using Neyman-Pearson lemma, we get that
$$
\eta^{**} = \left\{
\begin{array}{ll}
    1 & \mbox{if } \tilde{f}_{1}(Y;Y_{n+1}) \geq \tilde{f}_{-1}(Y;Y_{n+1}),\\
    -1 & \mbox{else,}
\end{array}
\right.
$$
minimizes $\mathbb{E}_{\pi^D} \mathbf{E}_{(\theta,\eta,\eta_{n+1})}|T(Y;Y_{n+1}) - \eta_{n+1}|$ over all functions of $(Y;Y_{n+1})$ with values in $[-1,1]$.
Using the independence of the rows of $(Y;Y_{n+1)}$, we have
$$ \tilde{f}_{\epsilon}(Y;Y_{n+1}) = \prod_{j=1}^{p}\frac{e^{-\frac{1}{2}L_{j}^{\top}(\Sigma^j_{\epsilon})^{-1}L_{j}}}{(2\pi)^{p/2}|\Sigma^j_{\epsilon}|}, $$
where $L_{j}$ is the $j$-th row of the matrix $(Y;Y_{n+1)}$ and $\Sigma^j_\epsilon=d_j\mathbf{I}_{n+1} + D_{jj} \eta_{\epsilon} \eta_{\epsilon}^{\top}$ (here $\eta_{\epsilon}$ is the binary vector such that $\eta_{\epsilon,i} = \eta_i$ for all $i=1,\dots,n$ and $\eta_{\epsilon,n+1} = \epsilon$). It is easy to check that $|\Sigma^j_{\epsilon}|=d_j^{n}(d_j+D_{jj}(n+1))$, hence it does not depend on $\epsilon$. A simple calculation leads to
\begin{align*}
(\Sigma^j_{\epsilon})^{-1} &= (1/d_j) \mathbf{I}_{n}- \frac{D_{jj}/d_j^2}{1+D_{jj}n/d_j}\eta_{\epsilon}\eta_{\epsilon}^{\top}\\
&=(1/d_j)\mathbf{I}_n-\frac{2\Delta^2\lambda/\sum_{i} d_i^2}{1+2n\lambda \Delta^2 d_j/\sum_{i} d_i^2}\eta_{\epsilon}\eta_{\epsilon}^{\top}\\
&= (1/d_j) \mathbf{I}_{n}- \frac{2\Delta^2\lambda/\sum_{i} d_i^2}{1+2(n\Delta^{2}d_j)/(\lambda r)}\eta_{\epsilon}\eta_{\epsilon}^{\top}.
\end{align*}
Hence 
\begin{equation}
\begin{aligned}
\frac{\tilde{f}_{1}(Y)}{\tilde{f}_{-1}(Y)}&= \prod_{j=1}^{p}e^{-\frac{1}{2}L_{j}^{\top}((\Sigma_{1}^{j})^{-1} - (\Sigma^{j}_{-1})^{-1})L_{j}} \nonumber \\
& = \prod_{j=1}^{p}\exp\left({\frac{2\Delta^2\lambda/\sum_{i} d_i^2}{1+(2n\Delta^{2}d_j)/(\lambda r)}L_{j,n+1}\sum_{k=1}^nL_{jk}\eta_{k}}\right) \nonumber \\
&= \exp\left(\frac{2\Delta^2 \lambda}{\sum_i d_i^2}\sum_{k=1}^n \eta_{k} \sum_{j=1}^{p}{\frac{L_{jk}L_{j,n+1}}{1+(2n\Delta^{2}d_j)/(\lambda r)} }\right)\\
&= \exp\left({\frac{2\Delta^2 \lambda}{\sum_i d_i^2} \langle Y_{n+1},\sum_{k=1}^n \eta_{k} \tilde{D} Y_{k} \rangle }\right) , \nonumber
\end{aligned}
\end{equation}
where $\tilde{D}=\diag\left(\frac{1}{1+(2n\Delta^{2}d_i)/(\lambda r)}\right)_{i=1,\dots,p}$.
We conclude that the optimal selector is given by
$$
\eta^{**}=\sign \left(Y_{n+1}^\top\left(\sum_{k=1}^n\eta_k \tilde{D}Y_k\right)\right)
$$
and that 
$$
\mathbf{E}(\R(\eta^{**}))=\mb P((\tilde{D}Y \eta)^\top Y_{n+1}<0)
$$
Let us denote $\hat{\theta} :=\frac{1}{n} \sum_{i=1}^n Y_i\eta_i = \theta + \xi$ where $\xi =\frac{1}{n} \sum_{i=1}^n W_i\eta_i$. Then
\[
\mathbf{E}(\R(\eta^{**}))= \E\left( \Phi^c \left( \frac{\langle \theta , \tilde{D}\hat{\theta} \rangle}{\sqrt{\hat{\theta}^\top\tilde{D} \Sigma \tilde{D}\hat{\theta}}}\right)\right).
\]
Observing that the eigenvalues of $\tilde{D}$ belong to $[1/3,1]$ and that $ \tilde{D} \Sigma \tilde{D} \succeq \Sigma /9$ in a PSD sense,  it is clear that
\[
\mathbf{E}(\R(\eta^{**}))\geq  \E\left( \Phi^c \left( \frac{3\langle \theta ,\tilde{D}\hat{\theta} \rangle}{\sqrt{\hat{\theta}^\top \Sigma \hat{\theta}}}\right)\right) \geq C\E\left(\exp\left(-c \frac{\|\theta\|^4+\langle \theta ,\tilde{D}\xi \rangle^2}{\hat{\theta}^\top \Sigma \hat{\theta}} \right)\right),
\]
for some $c,C>0$. Therefore
\[
\mathbf{E}(\R(\eta^{**})) \geq C\E\left(\exp\left(-c \frac{\|\theta\|^4+\langle \theta ,\tilde{D}\xi \rangle^2}{\xi^\top \Sigma \xi - 2 \|\theta\|^2 \lambda} \right)\right).
\]
Consider three events
\[
\mathbb{A}_1 = \{ \|\theta\|^2 \leq 2 \Delta^2 \lambda \},
\]
\[
\mathbb{A}_2 = \{ \xi^\top \Sigma \xi \geq r\lambda^2 /2n\},
\]
\[
\mathbb{A}_3 = \{ \langle \theta ,\tilde{D}\xi \rangle^2 \leq 2\Delta^4\lambda^2\}.
\]
Since $\Delta^4 \geq \Delta^2$ by assumption, we get
\[
\mathbf{E}(\R(\eta^{**})) \geq Ce^{-c'n\Delta^4/r} (1-\pi^{D}(\mathbb{A}^c_1)-\mathbf{P}(\mathbb{A}^c_2)-\mathbf{P}( \mathbb{A}^c_3)).
\]
Using Hanson-Wright inequality, we deduce that
\[
\pi^{D}(\mathbb{A}^c_1) + \mathbf{P}(\mathbb{A}^c_2) \leq 2e^{-cr} \leq 1/4,
\]
since $r/n \geq 10$. Moreover, we also have that
\[
\mathbf{P}(\mathbb{A}^c_3) \leq  \pi^{D}(e^{-c \Delta^4\lambda/\|\theta\|^2} )  \leq e^{-c\Delta^2} + \pi^{D}(\mathbb{A}^c_1) \leq 1/4,
\]
since $\Delta^2 \geq 100$. The proof is now complete.
\end{itemize}

\subsection{Proof of Theorem \ref{thm:upper}}
Let $\theta$ be a vector in $\mb R^p $. Without loss of generality we may assume that $\|\theta\|^4 \geq C_1 \theta^\top \Sigma \theta$ for some constant $C_1>0$ large enough, otherwise the result is trivial as the upper bound becomes of constant order. For the rest of the proof, we use the notation $\hat{\eta}:=\hat{\eta}_{\text{ave}}$. We start by observing that
\[
\mb{P}((\hat{\eta}(Y_{n+1})\neq \eta_{n+1})=\mb{P}\left(\left \langle \sum_{i=1}^n Y_i\eta_i,Y_{n+1}\eta_{n+1}\right\rangle<0\right).
\]
Let us denote $\hat{\theta} =\frac{1}{n} \sum_{i=1}^n Y_i\eta_i = \theta + \xi$ where $\xi =\frac{1}{n} \sum_{i=1}^n W_i\eta_i$.  We get the following upper bound
\[
\mb{P}((\hat{\eta}(Y_{n+1})\neq \eta_{n+1}) = \mb{P}(\langle \hat{\theta} , \theta + \eta_{n+1}W_{n+1}\rangle \leq 0) \leq \E\left(e^{-\frac{\langle \theta , \hat{\theta} \rangle^2}{2\hat{\theta}^\top \Sigma \hat{\theta}}}\right).
\]
 where we have conditioned on $\hat{\theta}$ to get the last inequality and used the fact that $\eta_{n+1}w_{n+1}$ has also i.i.d $1$-sub-Gaussian entries. Next, we have that 
\[
\langle \hat{\theta},\theta \rangle^2 = \left\langle \theta + \xi,\theta \right\rangle^2
\geq \frac{\|\theta\|^4}{2} - \left\langle \xi,\theta \right\rangle^2,\]
and that
\[
\hat{\theta}^\top \Sigma \hat{\theta} \leq 2 \|\Sigma\|_\infty \|\theta\|^2 + 2 \xi^\top\Sigma \xi.
\]
Hence
\[
\R(\hat{\eta}) \leq e^{-\frac{\|\theta\|^4- 2\left\langle \xi,\theta \right\rangle^2}{4 \|\Sigma\|_\infty \|\theta\|^2 + 4 \xi^\top\Sigma \xi}}.
\]
Let us define now the event
$$
\mathcal{A}=\{ \xi^\top \Sigma\xi\leq 3/2\Tr(\Sigma^2)/n + \|\Sigma\|^2_\infty \log(1/\delta)/n\}\cap\{4 \langle \xi,\theta\rangle^2\leq  \|\theta\|^4\}.
$$                                                              
Since $n\xi^\top \Sigma\xi=_d w^\top \Lambda^2w=\|\Lambda w\|^2, \mb E(\xi^\top \Sigma\xi)=\Tr(\Sigma^2)/n$ and $\sqrt{n}\xi^\top \theta=_d\theta^\top V \Lambda^{1/2} w$ where $w$ has independent $1-$sub-Gaussian entries.
Using Lemma \ref{lem:subG}, we get that
\[
\mb P(\mathcal{A}^c)\leq \delta  + e^{-cn\|\theta\|^4/\theta^\top \Sigma \theta} \leq \delta + e^{-c n},
\]
for some $c>0$ small enough.
Observe that, on event $\mathcal{A}$, we have
\[
e^{-c\frac{\|\theta\|^4- 2\left\langle \xi,\theta \right\rangle^2}{4 \|\Sigma\|_\infty \|\theta\|^2 + 4 \xi^\top\Sigma \xi}} \leq   
 \exp\left(-c\frac{\|\theta\|^4}{\theta^{\top}\Sigma \theta+ \frac{\Tr(\Sigma^2) + \|\Sigma\|^2_{\infty}\log(1/\delta)}{n} }\right).
\]
This show the result in probability. We next assume that $\|\theta\|^2 \geq \Delta^2 \|\Sigma\|_\infty$.
Replacing $\log(1/\delta)$ by $cn \Delta^2$, we get that
\[
\mb P(\mathcal{A}^c)\leq 2 e^{-cn\Delta^2} \leq 2e^{-c\frac{\Delta^4}{\Delta^2+\frac{r(\Sigma^2)}{n}}}.
\]
Moreover, on event $\mathcal{A}$, we have now
\[
e^{-c\frac{\|\theta\|^4- 2\left\langle \xi,\theta \right\rangle^2}{4 \|\Sigma\|_\infty \|\theta\|^2 + 4 \xi^\top\Sigma \xi}} \leq   
 \exp\left(-c\frac{\Delta^4}{\Delta^2+ \frac{r(\Sigma^2)}{n} }\right).
\]
Therefore, we see that
\[
\E(\R(\hat{\eta})) \leq \E\left(e^{-c\frac{\|\theta\|^4- 2\left\langle \xi,\theta \right\rangle^2}{4 \|\Sigma\|_\infty \|\theta\|^2 + 4 \xi^\top\Sigma \xi}}\mathbf{1}\{ \mathcal{A}\}\right) + 2e^{-c\frac{\Delta^4}{\Delta^2+\frac{r(\Sigma^2)}{n}}}.
\]
We conclude that 
\[
\E(\R(\hat{\eta})) \leq C\exp\left(-c''\frac{\Delta^4}{\Delta^2+\frac{r(\Sigma^2)}{n}}\right),
\]
for some $c'',C>0$.

\subsection{Proof of Proposition \ref{prop:LDA}}
Let $\theta$ be a vector in $\mb R^p $. For the rest of the proof, we use the notation $\hat{\eta}:=\hat{\eta}_{\text{LDA}}$. We start by observing that
\[
\mb{P}((\hat{\eta}(Y_{n+1})\neq \eta_{n+1})=\mb{P}\left(\left \langle \sum_{i=1}^n \Sigma^{-1/2}Y_i\eta_i,\Sigma^{-1/2}Y_{n+1}\eta_{n+1}\right\rangle<0\right).
\]
Let us denote $\hat{\theta} =\frac{1}{n} \sum_{i=1}^n \Sigma^{-1/2}Y_i\eta_i = \Sigma^{-1/2}\theta + \xi$ where $\xi =\frac{1}{n} \sum_{i=1}^n \Sigma^{-1/2}W_i\eta_i$.  We get the following lower bound under Gaussian noise
\begin{multline*}
\mb{P}((\hat{\eta}(Y_{n+1})\neq \eta_{n+1})
\\
= \mb{P}(\langle \hat{\theta} , \Sigma^{-1/2}\theta + \eta_{n+1}\Sigma^{-1/2}W_{n+1}\rangle \leq 0) \geq C\E\left(e^{-c\frac{\langle \Sigma^{-1/2}\theta , \hat{\theta} \rangle^2}{\|\hat{\theta}\|^2}}\right)
\end{multline*}
 for some $c,C>0$, where we have conditioned on $\hat{\theta}$ to get the last inequality and used the fact that $\eta_{n+1}w_{n+1}$ has i.i.d standard Gaussian entries. Next, we have that 
\[
\langle \hat{\theta},\Sigma^{-1/2}\theta \rangle^2 = \left\langle \Sigma^{-1/2}\theta + \xi,\Sigma^{-1/2}\theta \right\rangle^2
\leq 2\|\theta\|_{\Sigma}^4 + 2\left\langle \xi,\Sigma^{-1/2}\theta \right\rangle^2,\]
and that
\[
\|\hat{\theta}\|^2 \geq \|\theta\|^2_{\Sigma} + \|\xi\|^2 - 2\left| \theta^\top \Sigma^{-1/2}\xi\right|.
\]
Hence
\[
\mathbf{E}(\R(\hat{\eta})) \geq C\E\left(e^{-c\frac{2\|\theta\|_{\Sigma}^4 + 2\left\langle \xi,\Sigma^{-1/2}\theta \right\rangle^2}{ \|\theta\|^2_{\Sigma} + \|\xi\|^2 - 2\left| \theta^\top \Sigma^{-1/2}\xi\right|}}\right).
\]
Let us define now the event
$$
\mathcal{A}=\{ \|\xi\|^2 \geq p/(2n)\}\cap\left\{\left| \theta^\top \Sigma^{-1/2}\xi\right|\leq \|\theta\|_{\Sigma}/8\right\}.
$$                                                              
Since $n\|\xi\|^2=_d \| w\|^2, \mb E(\|\xi\|^2)=p/n$ and $\sqrt{n}\xi^\top\Sigma^{-1/2} \theta=_d \|\theta\|_{\Sigma}w_1$ where $w$ has independent standard Gaussian entries.
It is easy to see that
\[
\mb P(\mathcal{A}^c)\leq e^{-cn} + e^{-cp} \leq 1/2,
\]
for some $c>0$ small enough and $n,p$ large enough. Observe that, on event $\mathcal{A}$, we have
\[
e^{-c\frac{2\|\theta\|_{\Sigma}^4 + 2\left\langle \xi,\Sigma^{-1/2}\theta \right\rangle^2}{ \|\theta\|^2_{\Sigma} + \|\xi\|^2 - 2\left| \theta^\top \Sigma^{-1/2}\xi\right|}} \geq e^{-c\frac{2\|\theta\|_{\Sigma}^4 + 2\|\theta\|^2_{\Sigma}}{ \|\theta\|^2_{\Sigma} + p/(2n) - \|\theta\|_{\Sigma}/4}}.
\]
Hence,
\[
\E(\R(\hat{\eta})) \geq C\E\left(e^{-c\frac{2\|\theta\|_{\Sigma}^4 + 2\left\langle \xi,\Sigma^{-1/2}\theta \right\rangle^2}{ \|\theta\|^2_{\Sigma} + \|\xi\|^2 - 2\left| \theta^\top \Sigma^{-1/2}\xi\right|}} \mathbf{1}\{\mathcal{A}\}\right) \geq  C' e^{-c\frac{2\|\theta\|_{\Sigma}^4 + 2\|\theta\|^2_{\Sigma}}{ \|\theta\|^2_{\Sigma} + p/(2n) - \|\theta\|_{\Sigma}/4}}.
\]
for $C'=C/2$. If $\|\theta\|_{\Sigma} \geq 1/2$ the result is straightforward. Otherwise, if $\|\theta\|_{\Sigma} \leq 1/2$, then
\[
\frac{2\|\theta\|_{\Sigma}^4 + 2\|\theta\|^2_{\Sigma}}{ \|\theta\|^2_{\Sigma} + p/(2n) - \|\theta\|_{\Sigma}/4} \leq  \frac{C'' n}{p} \leq C''',
\]
and the bound is constant in this case. As a conclusion we get
\[
 \E(\R(\hat{\eta}_{\text{LDA}})  )\geq C' \exp\left(-c\frac{\|\theta\|^4_{\Sigma}}{\|\theta\|^2_{\Sigma}+ \frac{p}{n}}\right),
\]
for some $c,C'>0$.

\section{Proofs of regularization vs interpolation}

\subsection{Proof of Theorem~\ref{thm:interpolation1}}
Recall that
\begin{equation}\label{eq:reg}
\R(\hat{\eta}_{\lambda}) \leq e^{-\frac{\langle \theta , \hat{\theta}_{\lambda} \rangle^2}{2\hat{\theta}_\lambda^\top \Sigma \hat{\theta}_\lambda}},
\end{equation}
conditionally on $\hat{\theta}_\lambda$.
Observe that $\hat{\theta}_\lambda = \theta x/n + W^\top A_{\lambda}^{-1}\eta/n$ where $A_\lambda=\lambda I_n+Y^\top  Y/n$ and $x=\eta^\top  A_\lambda^{-1}\eta$. The risk is invariant by rescaling $\hat{\theta}_\lambda$ hence we rescale it by $n/x$. Hence, without loss of generality, we may assume that $\hat{\theta}_\lambda = \theta + W^\top H_{\lambda}^{-1}\eta/x$. 
Using Lemma~\ref{lem:3}, we have
\[
\hat{\theta}_\lambda = \left( I_p - W A_{\lambda}^{-1}W^\top /n \right)\theta + \frac{1+\eta^\top A_{\lambda}^{-1} W^\top \theta/n}{\eta^{\top} A_{\lambda}^{-1} \eta } W A_{\lambda}^{-1}\eta.
\]
On the one hand,
\[
|\langle \theta , \hat{\theta}_{\lambda} \rangle| \geq  \|\theta\|^{2} - \theta^{\top}W A_{\lambda}^{-1}W^\top\theta /n - \frac{|\eta^\top A_{\lambda}^{-1} W^\top \theta|}{\eta^{\top} A_{\lambda}^{-1} \eta} + \frac{\left(\theta^\top W A_{\lambda}^{-1}\eta\right)^2}{n\cdot \eta^{\top} A_{\lambda}^{-1} \eta}.
\]
On the other hand,
\begin{multline*}
\hat{\theta}_{\lambda}^\top \Sigma \hat{\theta}_{\lambda}\leq 2\Bigg(\left\|\Sigma^{1/2}\left( I_p - W A_{\lambda}^{-1}W^\top /n \right) \theta \right\|^2 
\\
+ \frac{2+2(\eta^\top A_{\lambda}^{-1} W^\top \theta/n)^2}{(\eta^{\top} A_{\lambda}^{-1} \eta)^2}\eta^\top  A_{\lambda}^{-1}W^\top \Sigma WA_{\lambda}^{-1} \eta\Bigg).
\end{multline*}
We will now control the numerator and denominator separately.
\begin{itemize}
\item Control of the numerator in \eqref{eq:reg}:

Let us denote $\theta_* := \pi_{k^{*}}\theta$ and $\bar{\theta} = \theta - \theta^{*}$. Observing that $I_p - W A_{\lambda}^{-1}W^\top /n$ is PSD with spectral norm less than $1$, we have
\begin{multline*}
\theta^\top\left( I_p - W A_{\lambda}^{-1}W^\top /n \right)\theta 
\\
\geq \bar{\theta}^\top\left( I_p - W A_{\lambda}^{-1}W^\top /n \right)\bar{\theta}/2 - \theta_{*}^\top\left( I_p - W A_{\lambda}^{-1}W^\top /n \right)\theta_*.
\end{multline*}
Therefore,
\[
|\langle \theta , \hat{\theta}_{\lambda} \rangle| \geq  \|\bar{\theta}\|^{2}/2 - \| A_{\lambda}^{-1/2}W^\top\bar{\theta}\|^2 /(2n) - \|\theta_*\|^2 - \frac{\|A_{\lambda}^{-1/2} W^\top \theta\|}{\sqrt{\eta^{\top} A_{\lambda}^{-1} \eta}}.
\]
Using Lemma \ref{lem:control_1} and Lemma \ref{lem:control_2}, we get that
\[
|\langle \theta , \hat{\theta}_{\lambda} \rangle| \geq \|\bar{\theta}\|^{2}/4 - C_1\frac{ \bar{\theta}^\top \Sigma \bar{\theta}}{\sum_{i>k^*}\lambda_i/n + \lambda} - C_2 \sqrt{\theta^\top \Sigma \theta},
\]
as $\|\theta_*\|^2 \leq \|\bar{\theta}\|^2/4$. Since $\bar{\theta}^\top \Sigma \bar{\theta} \leq \lambda_{k^{*}+1}\|\bar{\theta}\|^2$, then  
\[
|\langle \theta , \hat{\theta}_{\lambda} \rangle| \geq \|\bar{\theta}\|^{2}/8 - C_2 \sqrt{\theta^\top \Sigma \theta} \geq \|\theta\|^{2}/10 - C_2 \sqrt{\theta^\top \Sigma \theta}.
\]
If $\|\theta\|^{2}/10 \leq  2C_2 \sqrt{\theta^\top \Sigma \theta}$, then the bound in Theorem \ref{thm:interpolation1} is trivial. We conclude that 
\[
|\langle \theta , \hat{\theta}_{\lambda} \rangle| \geq C_3\|\theta\|^{2},
\]
for some $C_3>0$.
\item Control of the denominator in \eqref{eq:reg}:
\begin{multline*}
\hat{\theta}_{\lambda}^\top \Sigma \hat{\theta}_{\lambda}
\leq C\Bigg(
\theta^\top \Sigma \theta + \|\Sigma^{1/2} W A_{\lambda}^{-1}W^\top  \theta\|^2/n^2 
\\
+ \frac{2+2(\eta^\top A_{\lambda}^{-1} W^\top \theta/n)^2}{(\eta^{\top} A_{\lambda}^{-1} \eta)^2}\eta^\top  A_{\lambda}^{-1}W^\top \Sigma WA_{\lambda}^{-1} \eta\Bigg).
\end{multline*}
Using the same bound as for the numerator and the fact that 
$$\|W^\top \theta\|^{2} \leq Cn \theta^\top \Sigma \theta,$$ with probability $1-e^{-cn}$, we get further that
\[
\begin{aligned}
\hat{\theta}_{\lambda}^\top \Sigma \hat{\theta}_{\lambda}&\leq C\left(\theta^\top \Sigma \theta(1 + \|\Sigma^{1/2} W A_{\lambda}^{-1}\|_{\infty}^2/n )\right.\\
&\quad+ \left.\left(\left(\sum_{i>k^*}\lambda_i/n + \lambda\right)^2 + \theta^\top \Sigma \theta\right)\frac{\eta^\top  A_{\lambda}^{-1}W^\top \Sigma WA_{\lambda}^{-1} \eta}{n^2}\right).
\end{aligned}
\]
Using Lemma \ref{lem:subG} we have that with probability $1-\delta$
\[
\eta^\top  A_{\lambda}^{-1}W^\top \Sigma WA_{\lambda}^{-1} \eta \leq 3/2\Tr(A_{\lambda}^{-1}W^\top \Sigma WA_{\lambda}^{-1}) + \|A_{\lambda}^{-1}W^\top \Sigma WA_{\lambda}^{-1}\|_{\infty}\log(1/\delta).
\]
Hence, using Lemma \ref{lem:control_3} and Lemma \ref{lem:control_4}, we get further that with probability $1-\delta$
\[
\begin{multlined}
\eta^\top  A_{\lambda}^{-1}W^\top \Sigma WA_{\lambda}^{-1} \eta \leq C\left(k^{*} n +n \frac{\sum_{i>k^*} \lambda_{i}^{2}}{(\sum_{i>k^*}\lambda_i/n+\lambda)^{2}}\right)\\ + \left(k^{*} n + \frac{\sum_{i>k^*} \lambda_{i}^{2} + \lambda^2_{k^*+1}n}{(\sum_{i>k^*}\lambda_i/n+\lambda)^{2}}\right)\log(1/\delta),
\end{multlined}
\]
and that
\[
\|\Sigma^{1/2} W A^{-1}\|_{\infty}^2/n \leq C\left(k^{*}+\frac{\sum_{i>k^*} \lambda_{i}^{2}/n + \lambda^2_{k^*+1}}{(\sum_{i>k^*}\lambda_i/n+\lambda)^{2}}\right) \leq C(1+k^{*}).
\]
Notice also that
\[
\frac{\eta^\top  A_{\lambda}^{-1}W^\top \Sigma WA_{\lambda}^{-1} \eta}{n^2} \leq C_1+ C_2 \left(k^{*} + 1\right)\log(1/\delta)/n.
\]
We conclude that with probability $1- \delta_1 -\delta_2$
\[
\begin{multlined}
\hat{\theta}_{\lambda}^\top \Sigma \hat{\theta}_{\lambda}\leq C \left(\theta^\top \Sigma \theta( 1+k^{*})(1+\log(1/\delta_1)/n)+\frac{k^{*}\lambda^2_{k^*} + \sum_{i>k^*} \lambda_i^2}{n} \right. \\
\left. \quad + \frac{(k^{*}\lambda_{k^*}^2+\lambda_{k^*+1}^2)\log(1/\delta_2)}{n}\right).
\end{multlined}
\]
Taking $\delta_1=e^{-cn}$ yields that 
\[
\hat{\theta}_{\lambda}^\top \Sigma \hat{\theta}_{\lambda}\leq C \left(\theta^\top \Sigma \theta (1+k^{*}) + \frac{k^{*}\lambda^2_{k^*} + \sum_{i>k^*} \lambda_i^2}{n} + \frac{(k^{*}\lambda_{k^*}^2+\lambda_{k^*+1}^2)\log(1/\delta_2)}{n}\right).
\]
Above, we have used the fact that $k$ the smallest integer that satisfies $r_k(\Sigma)+\lambda n/\lambda_{k+1}>bn$ for $b\geq 1$.  We treat two cases. 
 (a) If $\Tr(\Sigma)/n + \lambda \geq b\|\Sigma\|_\infty$ then we can take $k^*=0$. Now if (b) $\Tr(\Sigma)/n + \lambda \leq b\|\Sigma\|_\infty$, then $k^* \geq 1$ and 
\[
\sum_{i > k^*} \lambda_i/n+\lambda \leq b \lambda_{k^*}.
\]
\end{itemize}

\subsection{Proof of Proposition~\ref{cor:interpolation1}}
Let $\theta \in \mathcal{C}_{k^*}(\Sigma)$. If $k^*=0$, the result follows immediately from Theorem \ref{thm:interpolation1}. We assume that $k^*>0$ (and hence $\sum_{i>k^*}\lambda_i/n + \lambda \leq b\|\Sigma\|_\infty$). Following the same steps as in the proof if Theorem \ref{thm:interpolation1}, we have for the numerator that
\[
|\langle \theta , \hat{\theta}_{\lambda} \rangle| \geq C_3\|\theta\|^{2},
\]
for some $C_3>0$ with probability $1-e^{-cn}$. As for the denominator,
\begin{multline*}
\hat{\theta}_{\lambda}^\top \Sigma \hat{\theta}_{\lambda}\leq 2\Bigg(\left\|\Sigma^{1/2}\left( I_p - W A_{\lambda}^{-1}W^\top /n \right) \theta \right\|^2 
\\
+ \frac{2+2(\eta^\top A_{\lambda}^{-1} W^\top \theta/n)^2}{(\eta^{\top} A_{\lambda}^{-1} \eta)^2}\eta^\top  A_{\lambda}^{-1}W^\top \Sigma WA_{\lambda}^{-1} \eta\Bigg)\\
\leq 2 \|\theta\|^2\|\Sigma\|_{\infty} + C\Big((\sum_{i>k^*}\lambda_i/n + \lambda)^2 \\
+ (\sum_{i>k^*}\lambda_i/n + \lambda)\|A_{\lambda}^{-1/2}W^\top \theta\|^2/n \Big)\frac{\eta^\top  A_{\lambda}^{-1}W^\top \Sigma WA_{\lambda}^{-1} \eta}{n^2},
\end{multline*}
with probability $1-e^{-cn}$.
Since $\|A_{\lambda}^{-1/2}W^\top\|^2_\infty/n \leq 1$, then
\[
\hat{\theta}_{\lambda}^\top \Sigma \hat{\theta}_{\lambda} \leq \|\theta\|^2\|\Sigma\|_{\infty} ,
\]
with probability $1-e^{-cn}$.
Hence
\[
\underset{\substack{\|\theta\| \geq \Delta \|\Sigma\|_\infty \\ \theta \in \mathcal{C}_{k^*}(\Sigma)
 }}{\sup} \E(\R(\hat{\eta}_{\lambda}) ) \leq C\exp\left(-c\frac{\Delta^4}{\Delta^2 + \frac{r(\Sigma^2)}{n}}\right) + e^{-cn}.
\]

\subsection{Proof of Theorem~\ref{thm:prolif}}

Arguing as in the proof of Lemma~\ref{lem:3}, we have that
\[
n(Y^\top Y)^{-1}\eta = \frac{\sqrt{n}}{\|\theta\| det} \left(A^{-1}u(1+u^\top A^{-1}v) - A^{-1}v u^\top A^{-1}u\right),
\]
where $A = W^\top W / n$, $u = \|\theta\| \eta / \sqrt{n}$ and $v=W^\top \theta / \sqrt{n\|\theta\|^2}$ and $det>0$.
Hence $e_1^{\top}(Y^\top Y)^{-1}\eta$ has the same sign as 
\[
e_{1}^\top (W^\top W )^{-1}\eta (1+ \eta^{\top} (W^\top W )^{-1}W^\top \theta) - e_{1}^\top (W^\top W )^{-1}W^\top \theta \eta^\top  (W^\top W )^{-1} \eta.
\]
Using Lemma~\ref{lem:inverse_canonical}, we also have that
\[
e_1(W^\top W)^{-1}\omega = \frac{ \omega_1 - W_{1}^\top \tilde{W}(\tilde{W}^\top \tilde{W})^{-1}\tilde{\omega}}{ \|W_1\|^2-W_{1}^\top \pi W_1}.
\]
Notice that $\pi=\tilde{W}(\tilde{W}^\top \tilde{W})^{-1}\tilde{W}^\top$ is a $p\times p$ projection matrix. Since $W$ is full rank with overwhelming probability, then $\|W_1\|^2-W_{1}^\top \pi W_1=W_1^\top(I_p-\pi)W_1>0$ with high probability as well.
\yiqiu{}

Hence we only need to show that the following expression is positive:
\[
(1-\eta_1 W_{1}^\top \tilde{W}(\tilde{W}^\top \tilde{W})^{-1}\tilde{\eta})(1+ \eta^{\top} (W^\top W )^{-1}W^\top \theta) - \eta_1W_{1}^\top(I_p - \pi)\theta \eta^{\top} (W^\top W)^{-1}\eta.
\]
We first use the bound 
\[
\eta^{\top} (W^\top W )^{-1}W^\top \theta \leq \sqrt{\eta A_0^{-1} \eta}\|A_0^{-1/2}W^\top \theta\|/n \leq C \frac{ \sqrt{\theta^\top \Sigma \theta}}{\sum_{i>k^*}\lambda_i/n },
\]
with probability $1-e^{-cn}$. Under the condition $\sqrt{\theta^\top \Sigma} \theta \leq 1/(2C) \sum_{i>k^*}\lambda_i/n$, we have that
\[
1+ \eta^{\top} (W^\top W )^{-1}W^\top \theta \geq 1/2.
\]
Next, observe that $\eta_1W_{1}^\top(I_p - \pi)\theta$ is $1$-sub-Gaussian with parameter $\|\Sigma^{1/2}(I_p - \pi)\theta\|$. Using the bound from the proof of Theorem \ref{thm:interpolation1}, we get, with probability $1-e^{-cn}$, that
\[
\|\Sigma^{1/2}(I_p - \pi)\theta\|^2 \leq  C\|\theta^\top \Sigma \theta\|^2 (1+k^{*}).
\]
The display above yields that with probability at least $1-C/n^2$, 
\[
|\eta_1W_{1}^\top(I_p - \pi)\theta \eta^{\top} (W^\top W)^{-1}\eta| \leq C\frac{ \sqrt{\theta^\top \Sigma \theta(1+k^*)\log(n)}}{\sum_{i>k^*}\lambda_i/n} \leq 1/4,
\]
under the conditions of the Theorem.
Finally, we have that $\eta_1 W_{1}^\top \tilde{W}(\tilde{W}^\top \tilde{W})^{-1}\tilde{\eta}$ is $1-$sub-Gaussian with parameter $\|\Sigma^{1/2} \tilde{W}(\tilde{W}^\top \tilde{W})^{-1}\tilde{\eta}\|$.
Using the proof of Theorem \ref{cor:interpolation1} we have
\[
\begin{multlined}
\|\Sigma^{1/2} \tilde{W}(\tilde{W}^\top \tilde{W})^{-1}\tilde{\eta}\|^2 \leq 
C\left(k^{*}/n  + \frac{\sum_{i>k^*} \lambda_{i}^{2}}{n(\sum_{i>k^*}\lambda_i/n)^{2}}\right)\\
+ \left(k^{*}/n + \frac{\sum_{i>k^*} \lambda_{i}^{2}/n^2 + \lambda^2_{k^*+1}/n}{(\sum_{i>k^*}\lambda_i/n)^{2}}\right)\log(1/\delta).
\end{multlined}
\]
Hence
\[
\|\Sigma^{1/2} \tilde{W}(\tilde{W}^\top \tilde{W})^{-1}\tilde{\eta}\|^2 \leq 
C\left(k^{*}\log(1/\delta)/n  +  \frac{(\sum_{i>k^*} \lambda_{i}^{2} + \lambda^2_{k^*+1}\log(1/\delta))/n}{(\sum_{i>k^*}\lambda_i/n)^{2}}\right).
\]
Therefore with probability $1-C/n^2$, 
\begin{multline*}
|\eta_1 W_{1}^\top \tilde{W}(\tilde{W}^\top \tilde{W})^{-1}\tilde{\eta}|
\\
\leq  C\left(\sqrt{k^{*}\log^2(n)/n} +  \frac{\sqrt{\sum_{i>k^*} \lambda_{i}^{2}\log(n)/n} + \lambda_{k^*+1}\log(n)/\sqrt{n}}{\sum_{i>k^*}\lambda_i/n}\right).
\end{multline*}
We conclude that under the conditions $k^{*}\log^2(n) \leq C n$ and $\sum_{i>k^*} \lambda_{i}^{2}n\log(n) \leq C (\sum_{i>k^*} \lambda_{i})^2$ the result follows.
We can now apply the union bound for $i=1,\dots,n$ and deduce that the final bound holds with probability at least $1-1/n$.

\subsection{Proof of Theorem~\ref{thm:robust}}
We will follow the same steps as in the proof of Theorem \ref{thm:interpolation1}. Recall that, under contamination, the new observations $\tilde{Y}_i$ are such that 
\[
\theta \eta_i + \tilde{\Sigma}^{1/2}w_i,
\]
where $w_i$ are i.i.d. standard normal vectors and $\tilde{\Sigma}:=\mathbf{I}_p + O$ where $O = \sum_{i \in R} O_i e_ie_i^\top$. Recall that
\begin{equation}\label{eq:reg1}
\R(\hat{\eta}_0) \leq e^{-\frac{\langle \theta , \hat{\theta}_{0} \rangle^2}{2\|\hat{\theta}_0\|^2}},
\end{equation}
conditionally on $\hat{\theta}_0$.
We denote $W_i=\tilde{\Sigma}^{1/2}w_i$. On the one hand, we have
\[
|\langle \theta , \hat{\theta}_0\rangle| \geq  \|\theta\|^{2} - \theta^{\top}W A_0^{-1}W^\top\theta /n - \frac{|\eta^\top A_0^{-1} W^\top \theta|}{\eta^{\top} A_0^{-1} \eta} + \frac{\left(\theta^\top W A_0^{-1}\eta\right)^2}{n\cdot \eta^{\top} A_0^{-1} \eta}.
\]
On the other hand,
\[
\| \hat{\theta}_0\|^2\leq 2\left(\left\| \theta \right\|^2 + \frac{2+2(\eta^\top A_0^{-1} W^\top \theta/n)^2}{(\eta^{\top} A_0^{-1} \eta)^2}\eta^\top  A_0^{-1}W^\top  WA_0^{-1} \eta\right).
\]
We will now control the numerator and denominator separately.
\begin{itemize}
\item Control of the numerator in \eqref{eq:reg1}:

Let us denote $\theta_* := \pi_{r}\theta$ and $\bar{\theta} = \theta - \theta^{*}$. Observing that $I_p - W A_0^{-1}W^\top /n$ is PSD with spectral norm less than $1$, we have
\begin{multline*}
\theta^\top\left( I_p - W A_0^{-1}W^\top /n \right)\theta
\\
\geq \bar{\theta}^\top\left( I_p - W A_0^{-1}W^\top /n \right)\bar{\theta}/2 - \theta_{*}^\top\left( I_p - W A_0^{-1}W^\top /n \right)\theta_*.
\end{multline*}
Hence
\[
|\langle \theta , \hat{\theta}_0 \rangle| \geq  \|\bar{\theta}\|^{2}/2 - \| A_0^{-1/2}W^\top\bar{\theta}\|^2 /(2n) - \|\theta_*\|^2 - \frac{\|A_0^{-1/2} W^\top \theta\|}{\sqrt{\eta^{\top} A_0^{-1} \eta}}.
\]
Using Lemma \ref{lem:control_1} and Lemma \ref{lem:control_2} replacing $k^{*}$ by $r$ (possible since $r \leq n/4$), we get that
\[
|\langle \theta , \hat{\theta}_0 \rangle| \geq \|\bar{\theta}\|^{2}/4 - C_1\frac{ \|\bar{\theta}\|^2}{(p-r)/n } - C_2 \| \theta\|\sqrt{p/n},
\]
as $\|\theta_*\|^2 \leq \|\bar{\theta}\|^2/4$. Since $\| \bar{\theta}\|^2 \leq \|\bar{\theta}\|^2$, then  
\[
|\langle \theta , \hat{\theta}_0 \rangle| \geq \|\bar{\theta}\|^{2}/8 - C_2 \|\theta\| \geq \|\theta\|^{2}/10 - C_2 \|\theta\|\sqrt{p/n}.
\]
Since $\|\theta\|^2 = \Omega(p/n)$, we conclude that 
\[
|\langle \theta , \hat{\theta}_0 \rangle| \geq C_3\|\theta\|^{2},
\]
for some $C_3>0$.
\item Control of the denominator in \eqref{eq:reg1}:
\[
\| \hat{\theta}_0\|^2\leq C\left(
\|\theta\|^2 + \frac{2+2(\eta^\top A_0^{-1} W^\top \theta/n)^2}{(\eta^{\top} A_0^{-1} \eta)^2}\eta^\top  A_0^{-1}W^\top   WA_0^{-1} \eta\right).
\]
Using the same bound as for the numerator and the fact that 
$$\|A_0^{-1/2}W^\top \theta\|^{2} \leq n\|\theta\|^2,$$  we get further that
\[
\|\hat{\theta}\|^2\leq
C\left(
\|\theta\|^2+ ((p/n)^2 + \|\theta\|^2p/n)\frac{\eta^\top  A_0^{-1}W^\top   WA_0^{-1} \eta}{n^2}\right).
\]
Using Lemma \ref{lem:subG} we have that with probability $1-\delta$
\[
\eta^\top  A _0^{-1}W^\top   WA _0^{-1} \eta \leq 3/2 n \Tr(A _0^{-1}) +  n\|A _0^{-1}\|_{\infty}\log(1/\delta).
\]
Hence, we get further that with probability $1-e^{-cn}$
\begin{equation*}
\eta^\top  A _0^{-1}W^\top WA _0^{-1} \eta \leq Cn^3/p.
\end{equation*}
We conclude that with probability $1- e^{-cn}$
\[
\|\hat{\theta}_0\|^2 \leq C (\|\theta\|^2+ p/n).
\]
\end{itemize}
\subsection{Proof of Proposition~\ref{prop:compare}}
Define $O$ such that $O = \lambda \sum_{i=1}^r e_ie_i^\top$ where $(e_1,\dots,e_p)$ is the canonical Euclidean basis of $\mathbb{R}^p$.   Under contamination, the new observations $\tilde{Y}_i$ have the same distribution as 
\[
\theta \eta_i + \tilde{W}_i,
\]
where $W_i$ are i.i.d. centered Gaussian vectors with covariance $\mathbf{I}_p + O$.
Let $\theta$ be a vector in $\mb R^p $. For the rest of the proof, we use the notation $\hat{\eta}:=\hat{\eta}_{\text{LDA}}=\hat{\eta}_{\text{ave}}$. Remember that
\[
\mb{P}((\hat{\eta}(Y_{n+1})\neq \eta_{n+1})=\mb{P}\left(\left \langle \sum_{i=1}^n \tilde{Y}_i\eta_i,Y_{n+1}\eta_{n+1}\right\rangle<0\right).
\]
Let us denote $\hat{\theta} =\frac{1}{n} \sum_{i=1}^n \tilde{Y}_i\eta_i = \theta + \xi$ where $\xi =\frac{1}{n} \sum_{i=1}^n \tilde{W}_i\eta_i$.  We get the following lower bound under Gaussian noise
\[
\mb{P}((\hat{\eta}(Y_{n+1})\neq \eta_{n+1}) = \mb{P}(\langle \hat{\theta} , \theta + \eta_{n+1}W_{n+1}\rangle \leq 0) \geq C\E\left(e^{-c\frac{\langle \theta , \hat{\theta} \rangle^2}{\|\hat{\theta}\|^2}}\right),
\]
 for some $c,C>0$, where we have conditioned on $\hat{\theta}$ to get the last inequality and used the fact that $\eta_{n+1}W_{n+1}$ has i.i.d standard Gaussian entries. Next, we have that 
\[
\langle \hat{\theta},\theta \rangle^2 = \left\langle \theta + \xi,\theta \right\rangle^2
\leq 2\|\theta\|^4 + 2\left\langle \xi,\theta \right\rangle^2,\]
and that
\[
\|\hat{\theta}\|^2 \geq \|\theta\|^2 + \|\xi\|^2 - 2\left| \theta^\top \xi\right|.
\]
Hence
\[
\mathbf{E}(\R(\hat{\eta})) \geq C\E\left(e^{-c\frac{2\|\theta\|^4 + 2\left\langle \xi,\theta \right\rangle^2}{ \|\theta\|^2 + \|\xi\|^2 - 2\left| \theta^\top \xi\right|}}\right).
\]
Let us define now the event
$$
\mathcal{A}=\{ \|\xi\|^2 \geq \lambda / 8\}\cap\left\{\left| \theta^\top \xi\right|\leq 10\sqrt{\lambda/n}\|\theta\|\right\}.
$$                                                   
It is easy to see that
\[
\mb P(\mathcal{A}^c) \leq 1/2.
\]
 Observe that, on event $\mathcal{A}$, we have for $n$ large enough
\[
e^{-c\frac{2\|\theta\|^4 + 2\left\langle \xi,\theta \right\rangle^2}{ \|\theta\|^2+ \|\xi\|^2 - 2\left| \theta^\top \xi\right|}} \geq e^{-c'\frac{\|\theta\|^4 + \|\theta\|^2\lambda/n}{ \|\theta\|^2+ \lambda}}.
\]
Hence, for $\lambda = n$ and $\|\theta\|^2 \leq \sqrt{n}$ we conclude that
\[
 \E(\R(\hat{\eta}_{\text{LDA}})  ) \wedge  \E(\R(\hat{\eta}_{\text{ave}})  )\geq C',
\]
for some $C'>0$.
\section{Auxiliary results (Algebra)}

\begin{lem}\label{lem1}
Let $u,v\in\mathbb{R}^n$ and $A \in \mathbb{R}^{n \times n}$ an invertible and symmetric matrix then
\begin{align*}
( uu^\top &+ uv^\top + vu^\top + A )^{-1} - A^{-1} \\
&=- \frac{(1-v^\top A^{-1} v)A^{-1}uu^\top A^{-1}+(1+u^\top A^{-1} v)A^{-1}(uv^\top + vu^\top) A^{-1}  - u^\top A^{-1} u A^{-1}vv^\top A^{-1}}{(1-v^\top A^{-1} v)u^\top A^{-1} u + (1+u^\top A^{-1} v)^2}.
\end{align*}
\end{lem}
\begin{proof}
We will use the Woodbury matrix identity
\[
(A+UCV)^{-1} = A^{-1} - A^{-1}U(C^{-1}+VA^{-1}U)^{-1}VA^{-1},
\]
with $U=(u \; v) \in \mathbb{R}^{n \times 2}$, $V=U^\top$ and $C= \begin{pmatrix}
1 & 1\\
1 & 0 
\end{pmatrix}$. We start by computing $C^{-1} = \begin{pmatrix}
0 & 1\\
1 & -1 
\end{pmatrix}$. It comes out that
\[
C^{-1}+VA^{-1}U = \begin{pmatrix}
u^\top A^{-1}u & 1+u^\top A^{-1}v\\
1+u^\top A^{-1}v & v^\top A^{-1}v -1 
\end{pmatrix}.
\]
Let us denote by $det := (1-v^\top A^{-1} v)u^\top A^{-1} u + (1+u^\top A^{-1} v)^2$, then
\[
(C^{-1}+VA^{-1}U)^{-1} = \frac{1}{det} \begin{pmatrix}
1 - v^\top A^{-1}v   & 1+u^\top A^{-1}v\\
1+u^\top A^{-1}v &   - u^\top A^{-1}u
\end{pmatrix}.
\]
Moreover
\[
U(C^{-1}+VA^{-1}U)^{-1}V = \frac{(1-v^\top A^{-1} v)uu^\top +(1+u^\top A^{-1} v)(uv^\top + vu^\top)  - u^\top A^{-1} u vv^\top }{det}.
\]
We conclude observing that
\[
uu^\top + uv^\top + vu^\top + A = A+UCV.
\]
\end{proof}

\begin{lem}\label{lem:2}
Let $u,v\in\mathbb{R}^n$ and $A \in \mathbb{R}^{n \times n}$ and an  invertible and symmetric matrix then
\[
( uu^\top + uv^\top + vu^\top + A )^{-1}u =  \frac{A^{-1}u(1+u^\top A^{-1}v) - A^{-1}v u^\top A^{-1}u}{(1-v^\top A^{-1} v)u^\top A^{-1} u + (1+u^\top A^{-1} v)^2},
\]
and 
\[
( uu^\top + uv^\top + vu^\top + A )^{-1}v =  \frac{A^{-1}v(1+u^\top A^{-1}v + u^\top A^{-1}u) - A^{-1}u (u^\top A^{-1}v + v^\top A^{-1}v)}{(1-v^\top A^{-1} v)u^\top A^{-1} u + (1+u^\top A^{-1} v)^2}.
\]
\end{lem}
\begin{proof}
Using Lemma \ref{lem1} we have
\begin{align*}
( uu^\top &+ uv^\top + vu^\top + A )^{-1} - A^{-1} \\
&=- \frac{(1-v^\top A^{-1} v)A^{-1}uu^\top A^{-1}+(1+u^\top A^{-1} v)A^{-1}(uv^\top + vu^\top) A^{-1}  - u^\top A^{-1} u A^{-1}vv^\top A^{-1}}{(1-v^\top A^{-1} v)u^\top A^{-1} u + (1+u^\top A^{-1} v)^2}.
\end{align*}
Hence
\[
( uu^\top + uv^\top + vu^\top + A )^{-1}u =  \frac{A^{-1}u(1+u^\top A^{-1}v) - A^{-1}v u^\top A^{-1}u}{(1-v^\top A^{-1} v)u^\top A^{-1} u + (1+u^\top A^{-1} v)^2}.
\]
The second part of the proof is also a straightforward computation.
\end{proof}
\begin{lem}\label{lem:3}
For any $\lambda \geq 0$, we have that
\[
\hat{\theta}_\lambda = \left( I_n - W A_{\lambda}^{-1}W^\top /n \right)\theta + \frac{1+\eta^\top A_{\lambda}^{-1} W^\top \theta/n}{\eta^{\top} A_\lambda^{-1} \eta } W A_\lambda^{-1}\eta,
\]
where $A_\lambda = \lambda I_n + W^\top W/n$.
We also have for $\lambda>0$ that
\[
\hat{\theta}_\lambda /c= (1-\eta^{\top}W^\top B _{\lambda}^{-1}W\eta/n^2 )B _{\lambda}^{-1}\theta + (1+\theta^\top B _{\lambda}^{-1}W \eta /n)B _{\lambda}^{-1}W\eta/n,
\]
where  $B_\lambda = \lambda I_p + W W^\top /n$ and $c>0$ some constant.
\end{lem}
\begin{proof}
Recall  that $\hat{\theta}_\lambda = \theta  + W H_{\lambda}^{-1}\eta/x$ where $H_\lambda=\lambda I_n+Y^\top Y/n$ and $x=\eta^\top  H_\lambda^{-1}\eta$. By denoting $w := \theta / \|\theta\|$, we get that
\[
H_{\lambda} = \frac{\|\theta\|^2 }{n} \eta\eta^\top + \frac{\|\theta\|}{n} (\eta (W^\top w)^\top  + (W^\top w)\eta^\top) + \lambda I_n + W^\top W/n.
\]
By choosing $u = \|\theta\| \eta / \sqrt{n}$, $v = W^\top w/\sqrt{n}$ and $A=\lambda I_n + W^\top W/n$ we get using Lemma \ref{lem:2} that
\[
H_{\lambda}^{-1}\eta = \frac{\sqrt{n}}{\|\theta\| det} \left(A^{-1}u(1+u^\top A^{-1}v) - A^{-1}v u^\top A^{-1}u\right),
\]
where $det = (1-v^\top A^{-1} v)u^\top A^{-1} u + (1+u^\top A^{-1} v)^2$. It comes out that
\[
\eta^\top H_{\lambda}^{-1}\eta = \frac{n u^\top A^{-1}u}{\|\theta\|^2 det} = \frac{\eta^\top A^{-1}\eta}{det} .
\]
We also get
\begin{align*}
W H_{\lambda}^{-1}\eta &= \frac{\sqrt{n}}{\|\theta\| det} \left(W A^{-1}u(1+u^\top A^{-1}v) - W A^{-1}v u^\top A^{-1}u\right) \\
&=\frac{1}{det} \left( W A^{-1} \eta (1+\eta^\top A^{-1}W^\top \theta/n ) - \frac{\eta^\top A^{-1} \eta}{n } W A^{-1} W^\top \theta \right).
\end{align*}
As a conclusion we get that
\[
\hat{\theta}_\lambda {det} = \left( I_n - W A^{-1}W^\top /n \right)\theta + \frac{1+\eta^\top A^{-1} W^\top \theta/n}{\eta^{\top} A^{-1} \eta } W A^{-1}\eta.
\]

On the other hand we also have 
\[
\hat{\theta}_\lambda =  \left(\theta \theta^\top + \theta (W\eta)^\top/n + W\eta \theta^\top/n + \lambda I_p + W W^\top /n \right)^{-1}(\theta + W\eta /n).
\]
By choosing $u = \theta $, $v = W\eta /n$ and $B=\lambda I_p + W W^\top /n$, we get using Lemma \ref{lem:2} that
\begin{align*}
\hat{\theta}_\lambda &= ( uu^\top + uv^\top + vu^\top + B )^{-1}(u+v)\\ &= \frac{1}{ det} \left(B^{-1}u(1-v^\top B^{-1}v) + B^{-1}v (1+u^\top A^{-1}v)\right).
\end{align*}
where $det = (1-v^\top B^{-1} v)u^\top B^{-1} u + (1+u^\top B^{-1} v)^2$. Hence 
\[
\hat{\theta}_\lambda det =(1-\eta^{\top}W^\top B^{-1}W\eta/n^2 )B^{-1}\theta + (1+\theta^\top B^{-1}W \eta /n)B^{-1}W\eta/n.
\]
\end{proof}

\begin{lem}\label{lem:inverse_canonical}
Assume that $p>n$. Let $W = (W_1 \; \tilde{W}) \in \mathbb{R}^{ p \times n}$ be a full rank matrix and $\omega = (\omega_1 \; \tilde{\omega}) \in \mathbb{R}^n$, then we have
\[
e_1(W^\top W)^{-1}\omega = \frac{ \omega_1 - W_{1}^\top \tilde{W}(\tilde{W}^\top \tilde{W})^{-1}\tilde{\omega}}{ \|W_1\|^2-W_{1}^\top \pi W_1},
\]
where $\pi = \tilde{W}(\tilde{W}^\top \tilde{W})^{-1}\tilde{W}^\top$.
\end{lem}
\begin{proof}
The r.h.s. is well-defined. Since $W$ is full rank, then $\tilde{W}^\top \tilde{W}$ is invertible and $\|W_1\|^2 > W_{1}^\top \pi W_1$. We will use the Schur complement formula (that holds as long as all matrix inverses exist).
\[
\begin{pmatrix}
A & B\\
B^{\top} & D 
\end{pmatrix}^{-1} = 
\begin{pmatrix}
A^{-1}+A^{-1}B(D-B^\top A^{-1}B)^{-1}B^\top A^{-1} & -A^{-1}B(D-B^\top A^{-1}B)^{-1}\\
-(D-B^\top A^{-1}B)^{-1}B^\top A^{-1} & (D-B^\top A^{-1}B)^{-1}.
\end{pmatrix}.
\]
Considering $A = \|W_1\|^2, B = W_{1}^\top \tilde{W}$ and $D = \tilde{W}^\top \tilde{W}$, then 
$$
W^\top W=\begin{pmatrix}
A & B\\
B^{\top} & D 
\end{pmatrix}
$$
By Sherman-Morrison formula we have that
\begin{align*}
(D-B^\top A^{-1}B)^{-1} &= \left(\tilde{W}^\top \tilde{W} - \frac{1}{\|W_1\|^2}\tilde{W}^\top W_1 W_{1}^\top \tilde{W}\right)^{-1}\\
&= (\tilde{W}^\top \tilde{W})^{-1} + \frac{\frac{1}{\|W_1\|^2}(\tilde{W}^\top \tilde{W})^{-1}\tilde{W}^\top W_1 W_{1}^\top \tilde{W}(\tilde{W}^\top \tilde{W})^{-1}}{1-\frac{W_{1}^\top \pi W_1}{\|W_1\|^2}}\\
&= (\tilde{W}^\top \tilde{W})^{-1}+\frac{(\tilde{W}^\top \tilde{W})^{-1}\tilde{W}^\top W_1 W_{1}^\top \tilde{W}(\tilde{W}^\top \tilde{W})^{-1}}{\|W_1\|^2-W_{1}^\top \pi W_1}\\
&=\frac{E}{\|W_1\|^2-W_{1}^\top \pi W_1}.\\
\end{align*}
where $E=(\tilde{W}^\top \tilde{W})^{-1}(\|W_1\|^2-W_{1}^\top \pi W_1)+(\tilde{W}^\top \tilde{W})^{-1}\tilde{W}^\top W_1 W_{1}^\top \tilde{W}(\tilde{W}^\top \tilde{W})^{-1}$.
Hence
\[
B(D-B^\top A^{-1}B)^{-1} = \frac{ W_{1}^\top \tilde{W}(\tilde{W}^\top \tilde{W})^{-1}}{1-\frac{W_{1}^\top \pi W_1}{\|W_1\|^2}}=\frac{\|W_1\|^2W_{1}^\top \tilde{W}(\tilde{W}^\top \tilde{W})^{-1}}{\|W_1\|^2-W_{1}^\top \pi W_1}.
\]
It comes out that
\[
A^{-1}+A^{-1}B(D-B^\top A^{-1}B)^{-1}B^\top A^{-1} = \frac{1}{\|W_1\|^2}\frac{1}{1-\frac{W_{1}^\top \pi W_1}{\|W_1\|^2}}=\frac{1}{\|W_1\|^2-W_{1}^\top \pi W_1},
\]
and that
\[
A^{-1}B(D-B^\top A^{-1}B)^{-1} = \frac{1}{\|W_1\|^2}\frac{ W_{1}^\top \tilde{W}(\tilde{W}^\top \tilde{W})^{-1}}{1-\frac{W_{1}^\top \pi W_1}{\|W_1\|^2}}=\frac{W_{1}^\top \tilde{W}(\tilde{W}^\top \tilde{W})^{-1}}{\|W_1\|^2-W_{1}^\top \pi W_1}.
\]
Hence
$$
(W^\top W)^{-1}=\frac{1}{\|W_1\|^2-W_{1}^\top \pi W_1}
\begin{pmatrix}
1& -W_{1}^\top \tilde{W}(\tilde{W}^\top \tilde{W})^{-1}\\
-(\tilde{W}^\top \tilde{W})^{-1}\tilde{W}^\top W_{1} & E
\end{pmatrix}
$$
and 
\[
e_1(W^\top W)^{-1}\omega = \frac{ \omega_1 - W_{1}^\top \tilde{W}(\tilde{W}^\top \tilde{W})^{-1}\tilde{\omega}}{ \|W_1\|^2-W_{1}^\top \pi W_1}.
\]
\end{proof}
\section{Auxiliary results (Probability)}
Let us write $\Sigma$ in its eigen-decomposition form $\Sigma=\sum_{i=1}^p \lambda_i v_iv_i^\top $ and decompose $W^\top W$ in the eigen basis of $\Sigma$. 
More specifically, let $z_i=W^\top  v_i/\sqrt{\lambda_i}$. Then $(z_i)_{1\leq i\leq n}$ is a basis  of $\mb R^{n\times n}$ if we assume $W$ to be full rank and $W^\top W=\sum_i \lambda_i z_i z_i^\top $. With our notations $A_{\lambda}= \frac{1}{n}W^\top W + \lambda I_n = \sum_i \lambda'_i z_i z_i^\top +\lambda I_n$ where $\lambda'_i = \lambda_i /n$. Let us also define 
$$
A_{-i}=\sum_{j\neq i}\lambda'_j z_j z_j^\top  + \lambda I_n,\quad A_k=\sum_{i>k}\lambda'_i z_iz_i^\top  + \lambda I_n,
$$
and similarly
$$ A^{0}_{-i}=A_{-i}-\lambda I_n,\quad  A^0_k=A_k-\lambda I_n.$$ 
We can apply Lemma 9 and 10 in \cite{bartlett2020benign} to $A^{0}_{-i}$ and $A^0_{k}$. Hence, and adding $\lambda$ to both sides, we get that there exists a constant $c>0$ such that with probability at least $1-2e^{-n/c}$:
\begin{itemize}
\item   For any $k\geq 0$
$$
\frac{1}{c}\sum_{i>k}\lambda'_i-c\lambda'_{k+1}n+\lambda\leq\lambda_n(A_k)\leq\lambda_1(A_k)\leq c\left(\sum_{i>k}\lambda'_i+\lambda'_{k+1}n\right)
+\lambda.
$$
\item $\forall i\geq 1$,
$$
\lambda_{k+1}(A_{-i})\leq\lambda_{k+1}(A_{\lambda})\leq\lambda_1(A_k)\leq c\left(\sum_{i>k}\lambda'_i+\lambda'_{k+1}n\right)+\lambda.
$$
\item $1\leq i\leq k$,
$$
\lambda_n(A_{\lambda})\geq\lambda_n(A_{-i})\geq \lambda_n(A_k)\geq \frac{1}{c}\sum_{i>k}\lambda'_i-c\lambda'_{k+1}n+\lambda.
$$
\end{itemize}
Recall that the effective rank is defined by $r_k(\Sigma):=\frac{\sum_{i>k}\lambda_i}{\lambda_{k+1}}$. Hence under the condition  $r_k(\Sigma) + \frac{n\lambda}{\lambda_{k+1}}\geq bn$ the above inequalities become 
\begin{itemize}
\item
for $k = k^{*}$
$$
\frac{1}{c}\sum_{i>k}\lambda_i/n+\lambda\leq\lambda_n(A_k)\leq\lambda_1(A_k)\leq c \sum_{i>k}\lambda_i/n
+\lambda.
$$
\item $\forall i\geq 1$,
$$
\lambda_{k+1}(A_{-i})\leq\lambda_{k+1}(A_{\lambda})\leq\lambda_1(A_k)\leq c\sum_{i>k}\lambda_i/n+\lambda.
$$
\item $1\leq i\leq k^{*}$,
$$
\lambda_n(A_{\lambda})\geq\lambda_n(A_{-i})\geq \lambda_n(A_k)\geq \frac{1}{c}\sum_{i>k}\lambda_i/n + \lambda.
$$
\end{itemize}
\begin{lem}\label{lem:subG}
Let $\xi\in \mathbb{R}^p$ be a random vector with $1$ sub-Gaussian   i.i.d entries and let $B\in \mathbb{R}^{p\times p}$ be a PSD matrix. Then with probability at least $1 - 2\delta$, we have
\[|\xi^\top B \xi - \Tr(B)| \leq 1/2 \Tr(B) + C \|B\|_{\infty} \log(1/\delta),\]
for some $C>0$.
\end{lem}
\begin{proof}
Notice that  $\mathbf{E}\xi^\top B\xi=\Tr(B).$ 
Using Hanson-Wright inequality \cite{rudelson2013hansonwright} we have that for any $t>0$,
\[
\mathbb{P}\left(|\xi^\top B \xi - \Tr(B)|>t\right) \leq 2 \exp \left[-c \min \left(\frac{t^{2}}{\|B\|_{F}^{2}}, \frac{t}{2\|B\|_\infty}\right)\right],
\]
for some $c>0$.
It follows that for $t\geq\Tr(B)/2$ we have $t/(2\|B\|_\infty)\leq t^2/(\Tr(B)\|B\|_\infty)\leq  t^2/\|B\|^2_F$. Therefore we can conclude for $C=2/c$.
\end{proof}
\begin{lem}\label{lem:control_1}
With probability at least $1-e^{-cn}$ we have
\[
\|A_{\lambda}^{-1/2} W^\top \theta\|^2 \leq C \frac{n \theta^\top \Sigma \theta}{\sum_{i>k^{*}}\lambda_i/n + \lambda}.
\]
for some $c, C>0$.
\end{lem}
\begin{proof}
We have that
\[
\| A_{\lambda}^{-1/2} W^\top \theta\|^2=\theta^\top WA^{-1}W^\top \theta \leq \| W^\top\theta\|_2^2 \|A^{-1}\|_\infty.
\]
Since $W^\top\theta$ has the same distribution as $\|\theta\|_{\Sigma}\cdot\xi$ where $\xi$ is a random vector with 1 sub-Gaussian with i.i.d entries, then using Lemma \ref{lem:subG} we get with probability at least $1-e^{-cn}$ 
\[
\|A_{\lambda}^{-1/2} W^\top \theta\|^2 \leq C \theta^\top \Sigma \theta n \|A_{\lambda}^{-1}\|_\infty.
\]
Hence we conclude that
\[
\|A_{\lambda}^{-1/2} W^\top \theta\|^2 \leq C \frac{n \theta^\top \Sigma \theta}{\sum_{i>k^*}\lambda_i/n + \lambda},
\]
since $1 / \|A_{\lambda}^{-1}\|_\infty=\lambda_n(A_{\lambda})\geq \sum_{i>k}\lambda_i/n+\lambda$ for $k= k^*$.
\end{proof}
\begin{lem}\label{lem:control_2}
There exist $c,C_1,C_2>0$ such that with probability at least $1-e^{-cn}$ we have
\[
\eta^{\top} A_{\lambda}^{-1} \eta \leq C_1\frac{n }{\sum_{i>k^*}\lambda_i/n + \lambda},
\]
and
\[
\eta^{\top} A_{\lambda}^{-1} \eta \geq C_2 \frac{n }{\sum_{i>k^*}\lambda_i/n + \lambda}.
\]
\end{lem}
\begin{proof}
The first inequality is straightforward observing that
\[
\eta^{\top} A_{\lambda}^{-1} \eta \leq n \|A_{\lambda}^{-1}\|_{\infty}.
\]
For the lower bound, we will the sub-Gaussian property of $\eta$. To lower bound $\Tr(A_{\lambda}^{-1})$ observe that
$$
\Tr(A_{\lambda}^{-1})=\sum_{i=1}^n\lambda_i(A_{\lambda})^{-1}\geq\sum_{i=k^{*}+1}^n\frac{1}{c\lambda_{k^{*}+1} r_{k^{*}}(\Sigma)/n+\lambda}\geq\frac{n/c}{\lambda_{k^*+1}r_{k^{*}}(\Sigma)/n+\lambda}.
$$
We conclude using Lemma \ref{lem:subG}. This is possible  by choosing $\delta = e^{-c'n}$ with $c'>0$ small enough since $n\|A_{\lambda}^{-1}\|_\infty \leq c'' \Tr(A_{\lambda}^{-1})$ for some constant $c''>0$.
\end{proof}
\begin{lem}\label{lem:control_3}
With probability at least $1-e^{-cn}$ we have
\[
\Tr( A_{\lambda}^{-1}W^\top \Sigma WA_{\lambda}^{-1} ) \leq c\left(k^{*} n +n \frac{\sum_{i>k^*} \lambda_{i}^{2}}{(\sum_{i>k^*}\lambda_i/n+\lambda)^{2}}\right),
\]
for some $c>0$.
\end{lem}
\begin{proof}
Let us denote by $C:= A^{-1}W^\top \Sigma WA^{-1}$. Using the rank one inverse formula we get
\begin{align*}
\Tr(C)&=\Tr(A^{-1}W^\top \Sigma WA^{-1})\\
&=\sum_{i=1}^n\lambda_i^2 z_i(\lambda'_i z_i z_i^\top +A_{-i})^{-2}z_i^\top \\
&=\sum_{i=1}^n\frac{\lambda_i^2 z_i^\top A_{-i}^{-2}z_i}{(1+\lambda'_iz_i^\top A_{-i}^{-1}z_i)^2}.
\end{align*}
Then for some $l\leq k^{*}$,
$$
\Tr(C)=\sum_{i=1}^l\frac{\lambda_i^2 z_i^\top A_{-i}^{-2}z_i}{(1+\lambda'_iz_i^\top A_{-i}^{-1}z_i)^2}+\sum_{i>l}\lambda_i^2 z_i^\top A_{\lambda}^{-2}z_i.
$$
Under the condition $r_k^{*}(\Sigma)+n\lambda/\lambda_{k^{*}+1}\geq bn$, there exists $c_1$ such that with probability at least $1-2e^{-n/c_1}$, for $i\leq k^*$, $\lambda_n(A_{-i})\geq \frac{1}{c}\sum_{i>k^{*}}\lambda_i/n +\lambda$. Hence
$$
z_i^\top  A_{-i}^{-2} z_i \leq \frac{c^{2}\|z_i\|^{2}}{\left(\sum_{i>k^{*}}\lambda_i/n+\lambda\right)^{2}}.
$$
Let $\mathscr{L}_i$ be the span of eigenvectors of $A_{-i}$ corresponding to the $n-k^{*}$ smallest eigenvalues. Then 
$$
z_i^\top  A_{-i}^{-1} z_i \geq\left(\Pi_{\mathscr{L}_{i}} z_i\right)^\top  A_{-i}^{-1} \Pi_{\mathscr{L}_{i}} z_i \geq \frac{\left\|\Pi_{\mathscr{L}_{i}} z_i\right\|^{2}}{1/c(\sum_{i>k}\lambda_i/n + \lambda)}
$$
Using the sub-Gaussian property of $z_i$, $k^* \leq n/c$ and the independence with $\mathscr{L}_{i}$, it comes out that with probability $1-e^{-cn}$, for all $i=1,\dots,n$ we have
\[
\left\|z_{i}\right\|^{2}  \leq 2n \quad \text{ and } \quad 
\left\|\Pi_{\mathscr{L}_{i}} z_{i}\right\|^{2}  \geq n/2.
\]
Hence the first sum can be bounded by 
$$
\sum_{i=1}^{k^*}\frac{\lambda_i^2 z_i^\top A_{-i}^{-2}z_i}{(1+\lambda'_iz_i^\top A_{-i}^{-1}z_i)^2}\leq c^4 \frac{n^2\sum_{i=1}^{k^{*}}\|z_i\|^2}{\|\Pi_{\mathscr{L}_{i}} z_{i}\|^{4}}\leq c' k^{*}  n.
$$
For the second sum, consider the same event where $\lambda_n(A_{\lambda})\geq \lambda_{k+1}r_k(\Sigma)/(nc_1)+\lambda$.
It comes out that
$$
\sum_{i>k^{*}}\lambda_i^2 z_i^\top A_{\lambda}^{-2}z_i\leq\frac{c_1^2\sum_{i>k^{*}}\lambda_i^2\|z_i\|^2}{(\sum_{i>k^{*}}\lambda_i/n+\lambda/c_1)^2}\leq \frac{c_5n\sum_{i>k^{*}}\lambda_i^2}{(\sum_{i>k^{*}}\lambda_i/n+\lambda /c_1)^2}.
$$
Therefore we have
$$
\Tr(C)\leq c\left(k^{*} n +n \frac{\sum_{i>k^{*}} \lambda_{i}^{2}}{(\sum_{i>k^{*}}\lambda_i/n+\lambda)^{2}}\right),
$$
for $0\leq k^* \leq n/c$ with  probability at least $1-e^{-cn}$.

\end{proof}
\begin{lem}\label{lem:control_4}
With probability at least $1-e^{-cn}$ we have
\[
\| A_{\lambda}^{-1}W^\top \Sigma WA_{\lambda}^{-1} \|_\infty \leq c\left(k^{*} n + \frac{\sum_{i>k^{*}} \lambda_{i}^{2} + \lambda^2_{k^{*}+1}n}{(\sum_{i>k^{*}}\lambda_i/n+\lambda)^{2}}\right),
\]
for some $c>0$.
\end{lem}
\begin{proof}
For the first half (elements $i \leq k^{*} $) we simply bound the spectral norm by the trace. For the second part, we have
\[
\left\|\sum_{i>k^{*}}\lambda_i^2 A_{\lambda}^{-1}z_iz_i^\top A_{\lambda}^{-1} \right\|_{\infty} \leq \|A_{\lambda}^{-1}\|_{\infty}^2 \left\|\sum_{i>k^{*}}\lambda_i^2 z_iz_i^\top  \right\|_{\infty}.
\]
Using previous arguments, we have with probability at least $1-e^{-cn}$ that
\[
\left\|\sum_{i>k^{*}}\lambda_i^2 z_iz_i^\top  \right\|_{\infty} \leq c\left(\sum_{i>k^*}\lambda^2_i+\lambda^2_{k^*+1}n\right).
\]
Hence
\[
\left\|\sum_{i>k^*}\lambda_i^2 A_{\lambda}^{-1}z_iz_i^\top A_{\lambda}^{-1} \right\|_{\infty} \leq \frac{c\left(\sum_{i>k^*}\lambda^2_i+\lambda^2_{k^*+1}n\right)}{(\sum_{i>k^*}\lambda_i/n+\lambda)^{2}}.
\]
We conclude that
\[
\| A_{\lambda}^{-1}W^\top \Sigma WA_{\lambda}^{-1} \|_\infty \leq c\left(k^{*} n + \frac{\sum_{i>k^*} \lambda_{i}^{2} + \lambda^2_{k^*+1}n}{(\sum_{i>k^*}\lambda_i/n+\lambda)^{2}}\right).
\]
\end{proof}

\end{document}